\documentclass[11pt]{article}
\usepackage[T1]{fontenc}
\usepackage[latin1]{inputenc}
\usepackage{amsmath,amssymb,euscript}
\usepackage{bbm}
\usepackage{graphicx}
\usepackage{psfrag}
\usepackage{mathrsfs}
\usepackage{enumerate}
\usepackage{natbib}

\newtheorem{thm}{Theorem}[section]
\newtheorem{lem}[thm]{Lemma}
\newtheorem{cor}[thm]{Corollary}

\newtheorem{defi}[thm]{Definition}
\newtheorem{prop}[thm]{Proposition}
\newtheorem{rem}[thm]{Remark}

\newenvironment{proof}{\noindent {\bf Proof \phantom{9}}}
{\hfill $\square$ \vspace{0.25cm}}

\def\be{\begin{eqnarray}}
\def\ee{\end{eqnarray}}
\def\ben{\begin{eqnarray*}}
\def\een{\end{eqnarray*}}

\numberwithin{equation}{section}
\numberwithin{figure}{section}

\def\be{\begin{eqnarray}}
\def\ee{\end{eqnarray}}

\newcommand{\RR}{\mathbb{R}}

\newtheorem{example}[thm]{Example}

\newtheorem{hypothesis}[thm]{Hypothesis}

\def\bit{\begin{itemize}}
\def\eit{\end{itemize}}
\reversemarginpar   
\def\bc{\begin{center}}
\def\ec{\end{center}}
\def\bthm{\begin{thm}}
\def\ethm{\end{thm}}
\def\bcor{\begin{cor}}
\def\ecor{\end{cor}}
\def\bprop{\begin{prop}}
\def\eprop{\end{prop}}
\def\blem{\begin{lem}}
\def\elem{\end{lem}}

\def\brem{\begin{rem}}
\def\erem{\end{rem}}

\def\bdes{\begin{description}}
\def\edes{\end{description}}

\def\beq{\begin{equation}}
\def\eeq{\end{equation}}
\def\benu{\begin{enumerate}}
\def\eenu{\end{enumerate}}
\def\beqar{\begin{eqnarray}}
\def\eeqar{\end{eqnarray}}
\def\beqarr{\begin{eqnarray*}}
\def\eeqarr{\end{eqnarray*}}

\def\RR{{\mathbb R}}  

\def\me{\medskip\noindent}
\def\bi{\bigskip\noindent}

\def\part{\partial}

\def\d#1dt{\frac{d#1}{dt}}    

\newcommand{\dst}{\displaystyle}
\newcommand{\mut}{\zeta}

\title{A rigorous model study of the adaptative dynamics of Mendelian
  diploids. }

\author{ Pierre Collet\thanks{CPHT, Ecole Polytechnique, 
CNRS UMR 7644, route de Saclay, 91128 Palaiseau Cedex-France; e-mail:
collet@cpht.polytechnique.fr}, 
Sylvie M\'el\'eard\thanks{CMAP, Ecole Polytechnique, CNRS, route de
Saclay, 91128 Palaiseau Cedex-France; e-mail:
sylvie.meleard@polytechnique.edu.}, 
Johan A.J. Metz \thanks{Institute of Biology \& Department of Mathematics, Leiden University,  \& NCB Naturalis, Leiden, Netherlands \& Ecology and Evolution Program, Institute of Applied Systems Analysis, Laxenburg. Austria; e-mail: j.a.j.metz@biology.leidenuniv.nl}}

\date{\today}

\begin{document}

\maketitle

\begin{abstract}
Adaptive dynamics so far has been put on a rigorous footing only for clonal inheritance. We extend this to sexually reproducing diploids, although admittedly still under the restriction of an unstructured population with Lotka-Volterra-like dynamics and single locus genetics (as in Kimura's 1965 infinite allele model). We prove under the usual smoothness assumptions, starting from a stochastic birth and death process model, that, when advantageous mutations are rare and mutational steps are not too large, the population behaves on the mutational time scale (the 'long' time scale of the literature on the genetical foundations of ESS theory) as a jump process moving between homozygous states (the trait substitution sequence of the adaptive dynamics literature). Essential technical ingredients are a rigorous estimate for the probability of invasion in a dynamic diploid population, a rigorous, geometric singular perturbation theory based, invasion implies substitution theorem, and the use of the Skorohod $M_1$ topology to arrive at a functional convergence result. In the small mutational steps limit this process in turn gives rise to a differential equation in allele or in phenotype space of a type referred to in the adaptive dynamics literature as 'canonical equation'. 
\end{abstract}

\bigskip


\pagebreak

\section{Introduction}
\label{sec:intro}

Adaptive dynamics (AD) aims at providing an ecology-based framework for scaling up from the micro-evolutionary process of gene substitutions to meso-evolutionary time scales and phenomena (also called long term evolution in papers on the foundations of ESS theory, that is, meso-evolutionary statics, cf \cite{E83,Eip,EFB98,EF01}). One of the more interesting phenomena that AD has brought to light is the possibility of an emergence of phenotypic diversification at so-called branching points, without the need for a geographical substrate \cite{M96,GKMM98,DD00}.  This ecological tendency may in the sexual case induce sympatric speciation \cite{DD99}. However,  a population subject to mutation limitation and initially without variation stays essentially uni-modal, closely centered around a type that evolves continuously, as long as it does not get in the neighborhood of a branching point. In this paper we focus on the latter aspect of evolutionary trajectories.

AD was first developed, in the wake of \cite{HS87,marrow-law-al-92,MNG92}, as a systematic framework at a physicist level of rigor  by Diekmann and Law  \cite{DL96} and by Metz and Geritz and various coworkers \cite{MNG92,M96,GKMM98}. The first two authors started from a Lotka-Volterra style birth and death process while the intent of the latter authors was more general, so far culminating in \cite{DMM08}. The details for general physiologically structured populations were worked out at a physicist level of rigor in \cite{DMM08} while the theory was put on a rigorous mathematical footing by Champagnat and M\'el\'eard and coworkers  \cite{CFM06,C06,MT09}, and recently also from a different perspective by Peter Jagers and coworkers \cite{KlebSagVatHacJag11}. All these papers deal only with clonal models. In the meantime a number of papers have appeared that deal on a heuristic basis with special models with Mendelian genetics (e.g. \cite{KG99,VD99,VD00,VDD06,PP06,PB08}), while the general biological underpinning for the ADs of Mendelian populations is described in \cite{Mip}. In the present paper we outline a mathematically rigorous approach along the path set out in \cite{CFM06,C06}, with proofs for those results that differ in some essential manner between the clonal and Mendelian cases. It should be mentioned though that just as in the special models in \cite{KG99,VD99,VD00,PP06,PB08} and contrary to the treatment in \cite{Mip}  we deal still only with the single locus infinite allele case (cf Kimura \cite{MK65}), while deferring the infinite loci case to a future occasion.   

Our reference framework is a diploid  population in which each individual's ability to survive and reproduce depends only on a quantitative phenotypic trait determined by its genotype, represented by the types of  two alleles on a single locus.  Evolution of the trait distribution in the population results from three basic mechanisms: \emph{heredity}, which transmits traits to new offsprings thus ensuring the extended existence of a trait distribution, \emph{mutation}, generating novel variation in the trait values in the population, and  \emph{selection} acting on these trait values as a result of trait dependent differences in fertility and mortality. Selection is made  frequency dependent by the competition of individuals for limited resources, in line with the general ecological spirit of AD. Our goal is to capture in a simple manner the interplay between these different mechanisms.



\section{The Model}
\label{sec:model}

We consider a  Mendelian population and a hereditary trait that is determined by the two alleles on but a single locus with many possible alleles (the infinite alleles model of Kimura \cite{MK65}).  These alleles are characterized by an allelic trait $u$. Each individual $i$ is thus characterized by its two  allelic trait values $(u^i_{1}, u^i_{2})$, hereafter referred to as its genotype, with corresponding phenotype $\phi(u^i_{1}, u^i_{2})$, with $\phi :  \mathbb{R}^m \rightarrow  \mathbb{R}^n$. In order to keep the technicalities to a minimum we shall below proceed on the assumption that $n=m=1$. In the Discussion we give a heuristic description of how the extension to general $n$ and $m$ can be made. When we are dealing with a fully homozygous population we shall refer to its unique allele as $A$ and when we consider but two co-circulating alleles we refer to these as $A$ and $a$.

\noindent We make the standard assumptions that $\phi$ and all other coefficient functions are smooth and that there are no parental effects, so that $\phi(u_1,u_2)=\phi(u_2,u_1)$, which has as immediate consequence that if $u_a=u_A+\zeta$, $| \zeta| \ll 1$,  then $\phi(u_A,u_a)=\phi(u_A,u_A)+\partial_2\phi(u_A,u_A)\zeta+\mathrm{O}(\zeta^2)$ and $\phi(u_a,u_a)=\phi(u_A,u_A)+2\partial_2\phi(u_A,u_A)\zeta+\mathrm{O}(\zeta^2)$, i.e., the genotype to genotype map is locally additive, $\phi(u_A,u_a) \approx \left(\phi(u_A,u_A)+\phi(u_a,u_a)\right)/2$, and the same holds good for all quantities that smoothly depend on the phenotype.

\begin{rem}\rm
The biological justification for the above assumptions is that the evolutionary changes that we consider are not so much changes in the coding regions of the gene under consideration as in its regulation. Protein coding regions are in general preceded by a large number of relatively short regions where all sorts of regulatory material can dock. Changes in these docking regions lead to changes in the production rate of the gene product. Genes are more or less active in different parts of the body, at different times during development and under different micro-environmental conditions. The allelic type $u$ should be seen as a vector of such expression levels. The genotype to phenotype map $\phi$ maps these expression levels to the phenotypic traits under consideration. It is also from this perspective that we should judge the assumption of smallness of mutational steps $\zeta$: the influence of any specific regulatory site among its many colleagues tends to be relatively minor. 
\end{rem}

The individual-based microscopic model from which we start is a stochastic birth and death process, with density-dependence through additional deaths from ecological competition, and Mendelian reproduction with mutation. We assume that the population's size scales with a parameter $K$  tending to infinity while the effect of the interactions between individuals scales with ${1\over K}$. This allows taking limits in which we count  individuals weighted with ${1\over K}$. As an interpretation think of individuals that live in an area of size $K$ such that the individual effects get diluted with area, e.g. since individuals compete for living space, with each individual taking away only a small fraction of the total space,  the probability of finding a usable bit of space is proportional to the relative frequency with which such bits are around.

\subsection{Model setup}
The allelic trait space ${\cal U}$ is assumed to be a closed and bounded interval of $\mathbb{R}$.  Hence the phenotypic trait space is compact.  For any $(u_{1}, u_{2})\in{\cal U}^2$, we introduce the following demographic parameters, which are all assumed to be smooth functions of the allelic traits and thus bounded. Moreover, these parameters are assumed to depend in principle on the allelic traits through the intermediacy of the phenotypic trait.  Since the latter dependency  is  symmetric, we assume that all coefficient functions defined below are symmetric in the allelic traits. 

\begin{description}
\item[$f(u_{1}, u_{2})\in\mathbb{R}_+$]: the per capita birth rate (fertility) of an  individual
  with genotype $(u_{1}, u_{2})$.
\item[$D(u_{1}, u_{2})\in\mathbb{R}_+$]: the background death rate of an individual with genotype $(u_{1}, u_{2})$.

\item[$K\in\mathbb{N}$]: a parameter scaling  the per capita impact on resource density and through that the population size.
  
 \item[$\frac{C((u_{1}, u_{2}),( v_{1}, v_{2}))}{K} \in\mathbb{R}_+$]:  the competitive effect felt by an individual with genotype $(u_{1}, u_{2})$ from an individual with genotype $(v_{1}, v_{2})$. The function $C$ is customarily referred to as competition kernel.
  
  \item $\mu_{K}\in\mathbb{R}_+$: the
  mutation probability  per birth event (assumed to be independent of the genotype). The idea is that $\mu_K$ is made appropriately small when we let $K$ increase. 
  
  \item $\sigma>0$: a parameter scaling the mutation amplitude.
  
\item $m_{\sigma}(u,h)dh = {1\over \sigma} m(u,{h\over \sigma})dh$:   the mutation law of a mutant allelic trait $\, u+h$ from an individual with allelic trait $u$,  with  $m(u,h)dh$ a probability measure with support   $ [-1,1]\cap\{h \,  | \,   u+h\in{\cal U}\}$. As a result the support of $m_{\sigma}$ is of size $\leq 2\sigma$.

\end{description}


\smallskip
 \noindent {\bf Notational convention}:
When only two alleles $A$ and $a$ co-circulate, we will  use the shorthand:
\ben &f_{AA}=f(u_{A},u_{A})\ ;\ f_{Aa}=f(u_{A},u_{a})\ ;\
f_{aa}=f(u_{a},u_{a})\ ;\ D_{AA} = D(u_{A},u_{A})\ ;\\  
& C((u_{A},u_{a}),(u_{A},u_{A})) = C_{Aa,AA};\;\; \mathrm{etc.}\een

\bi
To keep things simple we take our model organisms to be hermaphrodites which in their female role give birth at rate $f$ and in their male role have probabilities proportional to $f$ to act as the father for such a birth.

\noindent We consider, at any time $t\geq 0$, a finite number
$N_t$ of individuals, each of them with genotype in ${\cal
U}^2$. Let us denote by $(u^1_{1},u^1_{2}),\ldots,(u_{1}^{N_t},u_{2}^{N_t})$ the genotypes of these individuals. The state of the population at time $t\geq 0$,
rescaled by $K$, is described by the finite point measure on
${\cal U}^2$
\begin{equation}
  \label{eq:nu_t}
  \nu^{\sigma,K}_t={1\over K}\sum_{{i=1}}^{N_{t}} \delta_{(u^i_{1}, u^i_{2})}, 
\end{equation}
where $\delta_{(u_{1},u_{2})}$ is the Dirac measure at $(u_{1},u_{2})$. 

\noindent Let $ \langle\nu,g\rangle$ denote the integral of the measurable
function $g$ with respect to the measure $\nu$ and $\mathrm{Supp}(\nu)$ the support of the latter.
Then $\: \langle\nu^{\sigma,K}_t,{\bf 1}\rangle=\frac{N_t}{K}$ and
for any $(u_{1},u_{2})\in {\cal U}^2$, the positive number $\langle\nu^{\sigma,K}_t,{\bf
1}_{\{(u_{1},u_{2})\}}\rangle$ is called the {\bf density} at time $t$ of
genotype $(u_{1},u_{2})$.

\noindent  Let ${\cal M}_F$ denote the set of finite nonnegative
measures on ${\cal U}^2$, equipped with the weak topology, and
define
\begin{equation*}
  {\cal M}^K=\left\{\frac{1}{K}\sum_{i=1}^n\delta_{(u_1^i,u_2^i) }:n\geq 0,\
 (u^1_{1},u^1_{2}) ,\ldots, (u^n_{1},u^n_{2})\in{\cal U}^2\right\}.
\end{equation*}

\noindent An individual with genotype $(u_{1},u_{2})$ in the population
$\nu^{\sigma,K}_t$ reproduces with an individual  with genotype $(u^j_{1},u^j_{2})$ at a rate $f(u_{1},u_{2}) {f(u^j_{1},u^j_{2})\over K\langle \nu^{\sigma,K}, f\rangle}$.

\noindent  With probability $1- \mu_K(u_{1},u_{2})$  reproduction follows the Mendelian rules, with a
newborn getting a genotype with coordinates that are sampled at random from each parent.

At reproduction mutations occur with probability  $\mu_K(u_{1},u_{2})$ and then change one of the two allelic traits of the newborn from $u$ to $u+h$ with $h$ drawn from $m_\sigma(u,h)dh$.

Each individual  dies at rate
\begin{eqnarray*}
  D(u_{1},u_{2})+ C*\nu^{\sigma,K}_t(u_{1},u_{2}) = D(u_{1},u_{2}) +
  \frac{1}{K}\sum_{j=1}^{N_t}C((u_{1},u_{2}); (u^j_{1},u^j_{2})).
\end{eqnarray*}
The competitive effect of individual $j$ on an  individual $i$ is described by an increase of ${C((u_{1}^i,u_{2}^i); (u^j_{1},u^j_{2}))\over K}$ of the latter's death rate. The parameter $K$ scales the strength of competition: the larger $K$, the less individuals interact. This decreased interaction goes hand in hand with a larger population size, in such a way that densities stay well-behaved. Appendix \ref{compkern} summarizes the long tradition of and supposed rationale for the representation  of competitive interactions by competition kernels.



 For measurable functions $F: \mathbb{R}\rightarrow \mathbb{R}$ and $g: {\cal U}^2\rightarrow \mathbb{R}$,  $g$  symmetric, let us define the function $F_{g}$ on ${\cal M}^K$ by $F_{g}(\nu) = F(\langle \nu, g\rangle)$.  

\noindent
For a genotype $(u_{1},u_{2})$ and a point measure $\nu$,  we define the Mendelian reproduction operator
\be 
\label{op-reproduction}
&&AF_{g}(\nu,u^i_{1},u^i_{2},u^j_{1},u^j_{2}) = {1\over 4}\bigg\{F\Big(  \langle \nu, g\rangle+{1\over K} g(u^i_{1},u^j_{1})\Big)+ F\Big(  \langle \nu, g\rangle+{1\over K} g(u^i_{1},u^j_{2})\Big) \nonumber\\
&&\hskip 2cm +F\Big(  \langle \nu, g\rangle+{1\over K} g(u^i_{2},u^j_{1})\Big) +F\Big(  \langle \nu, g\rangle+{1\over K} g(u^i_{2},u^j_{2})\Big)\bigg\}  - F_{g}(\nu),
 \ee
and for $m(u,h)dh$ a measure on $\mathbb{R}$ parametrized by $u$, we define the Mendelian reproduction-cum-mutation operator
\be
\label{op-mutation}
&MF_{g}(\nu, u^i_{1},u^i_{2},u^j_{1},u^j_{2})\\
&={1\over 8}\int \Big\{\Big(F\Big(  \langle \nu, g\rangle+{1\over K} g(u^i_{1}+h,u^j_{1})\Big)+F\Big(  \langle \nu, g\rangle+{1\over K} g(u^i_{1}+h,u^j_{2})\Big)\Big)\, m_{\sigma}(u^i_{1},h)  \nonumber\\
&+\Big(F\Big(  \langle \nu, g\rangle+{1\over K} g(u^i_{2}+h,u^j_{1})\Big) +F\Big(  \langle \nu, g\rangle+{1\over K} g(u^i_{2}+h,u^j_{2})\Big)\Big)m_{\sigma}(u^i_{2},h) \nonumber \\
&+\Big(F\Big(  \langle \nu, g\rangle+{1\over K} g(u^i_{1},u^j_{1}+h)\Big) +F\Big(  \langle \nu, g\rangle+{1\over K} g(u^i_{2},u^j_{1}+h)\Big)\Big)m_{\sigma}(u^i_{2},h)  \nonumber \\
&+\Big(F\Big(  \langle \nu, g\rangle+{1\over K} g(u^i_{1},u^j_{2}+h)\Big) +F\Big(  \langle \nu, g\rangle+{1\over K} g(u^i_{2},u^j_{2}+h)\Big)\Big)m_{\sigma}(u^i_{2},h) \Big\}
\, dh
 - F_{g}(\nu).\nonumber\\
 \
\ee


\noindent  The process $(\nu^{\sigma,K}_t,t\geq 0)$ is a ${\cal M}^K$-valued Markov process with infinitesimal generator defined for
any bounded measurable functions $F_{g}$ from ${\cal M}^K$ to
$\mathbb{R}$

 \noindent
  and $\nu =\frac{1}{K}\sum_{i=1}^n \delta_{(u_1^i,u_2^i)} $  by

\be
  L^KF_{g}(\nu) & =& \sum_{i=1}^n \Big( D(u^i_{1},u^i_{2})+ C*\nu^{\sigma,K}_t(u^i_{1},u^i_{2})\Big) \left(F\Big(  \langle \nu, g\rangle - {1\over K} g(u^i_{1},u^i_{2})\Big) - F_{g}(\nu)\right)\nonumber\\
&&+  \sum_{i=1}^n (1 -  \mu_{K}(u_1^i,u_2^i))\sum_{j=1, j\neq i}^n  f(u^i_{1},u^i_{2}) {f(u^j_{1},u^j_{2})\over K \, \langle \nu, f\rangle} \,AF_{g}(\nu,u^i_{1},u^i_{2},u^j_{1},u^j_{2})\nonumber\\
&+&  \sum_{i=1}^n \mu_{K}(u_1^i,u_2^i)\sum_{j=1, j\neq i}^n  f(u^i_{1},u^i_{2}) {f(u^j_{1},u^j_{2})\over K \, \langle \nu, f\rangle}\, MF_{g}(\nu, u^i_{1},u^i_{2},u^j_{1},u^j_{2})
.
   \label{eq:generator-renormalized-IPS}
\ee

\noindent  The first term  describes  the deaths,  the second term describes the births without mutation and the third term describes  the
 births with mutations. (We neglect the occurrence of multiple mutations in one zygote, as those unpleasantly looking terms will become negligible anyway when $\mu_K$ goes to zero.)  The density-dependent non-linearity of the death
term models the competition between individuals and makes selection frequency dependent.  \medskip

\noindent Let us denote by~(A) the following three assumptions
\begin{description}
\item[\textmd{(A1)}]  The functions $f$, $D$, $\mu_{K}$ and $C$ are smooth functions and thus bounded since ${\cal U}$ is compact.Therefore there exist $\bar{f},\bar{D}, \bar{C}<+\infty$ such that
  \begin{equation*}
   0\leq f(\cdot)\leq\bar{f},\quad 0\leq D(\cdot)\leq\bar{D},     \quad 0\leq C(\cdot,\cdot)\leq\bar{C}.
  \end{equation*}
  \item[\textmd{(A2)}] $r(u_{1},u_{2})=f(u_{1},u_{2})-D(u_{1},u_{2})>0$ for any $(u_{1},u_{2})\in {\cal U}^2$, and there exists
  $\underline{C}>0$ such that
$ \ \underline{C}\leq C(\cdot,\cdot)$.

\item[\textmd{(A3)}]  For any $\sigma>0$, there exists a function
$\bar{m}_{\sigma}:\mathbb{R}\rightarrow\mathbb{R}_+$,  $\int \bar{m}_{\sigma}(h)dh<\infty$,  such
  that $m_{\sigma}(u,h)\leq \bar{m}_{\sigma}(h)$ for any $u\in{\cal U}$ and
  $h\in\mathbb{R}$.

\end{description}
\noindent For fixed $K$, under~(A1) and~(A3) and assuming that
$\mathbb{E}(\langle\nu^{\sigma,K}_0,\mathbf{1}\rangle)<\infty$, the existence
and uniqueness in law of a process on $\mathbb{D}(\mathbb{R}_+, {\cal
  M}^K)$ with infinitesimal generator $L^K$ can be adapted from the one in Fournier-M\'el\'eard \cite{FM04} or \cite{CFM06}. The process can be constructed as solution of a stochastic differential equation driven by point Poisson measures describing each jump event. 
  Assumption~(A2) prevents the
population from exploding or going  extinct too fast.

\section{The short term large population and rare mutations limit: how selection changes allele frequencies}\label{short term}

In this section we study the large population and rare mutations  approximation of the process described above, when $K$ tends to infinity and $\mu_{K}$ tends to zero. The limit becomes deterministic and continuous and the mutation events disappear. 

The proof of the following theorem can be adapted from \cite{FM04}.

\begin{thm} \label{largepop}
When $K$ tends to infinity and if $\nu_{0}^K $ converges in law to a
deterministic measure $\nu_{0}$, then the process $(\nu^{\sigma,K})$
converges in law to the deterministic continuous  measure-valued
function $(\nu_{t}, t\geq 0)$ solving   
\begin{eqnarray*}
&\langle \nu_{t},g\rangle = \langle \nu_{0},g\rangle + \int_{0}^t \bigg\{- \langle \nu_{s}, (D+ C*\nu_{s})g\rangle+ \langle \nu_{s}\otimes \nu_{s}, {f(u_{1},u_{2}) f(v_{1},v_{2})\over 4 \langle \nu_{s}, f\rangle}\big(g(u_{1},v_{1}) + g(u_{1}, v_{2}) \\
&\hskip2cm + g(u_{2}, v_{1})+g(u_{2}, v_{2})\big) \rangle\bigg\}ds . 
\end{eqnarray*}
\end{thm} 

\noindent Below we have a closer look at  the specific cases of  genetically mono- and dimorphic initial conditions.

\subsection{Monomorphic populations}

Let us first study the dynamics of a fully homozygote population with genotype $(u_{A},u_{A})$ corresponding to a unique allele $A$ and genotype $AA$. Assume that the initial condition is $N^K_{0}\delta_{(u_{A},u_{A})}$, with ${N^K_{0}\over K}$ converging to a deterministic number $n_{0}>0$ when $K$ goes to infinity. 
 
 \noindent In that case the population process is $N^K_{t }\delta_{(u_{A},u_{A})}$ where $N^K_{t }$ is a logistic birth and death process 
 with birth rate $f_{AA}=f(u_{A},u_{A})$ and death rate 
 $D_{AA} + {C_{AA,AA}\over K}\, N^K_{t}$. The process  $({N^K_{t}\over K}, t\geq 0)$ converges in law when $K$ tends to infinity to the solution $(n(t), t\geq 0)$  of the logistic equation 
 \be\label{logistic} {d n\over dt}(t) = n({t})\,(f_{AA} - D_{AA} - C_{AA,AA}\, n({t})),\ee with initial condition $n(0)=n_{0}$.
 This equation  has a unique stable equilibrium equal to the carrying capacity:
 \be
 \label{capacity}\bar n_{AA}= {f_{AA}- D_{AA}\over C_{AA,AA}}.\ee
 
 \subsection{Genetic dimorphisms}
 
Let us now assume that there are two alleles $A$ and $a$  in the population (and no mutation). Then the initial population has the three genotypes  $AA$, $Aa$ and $aa$.  We use $(N^K_{AA,t}, N^K_{Aa,t}, N^K_{aa,t})$ to denote the respective numbers of individuals with genotype $AA$, $Aa$ and $aa$ at time $t$, and 
$(N_{AA},N_{Aa},N_{aa})$ 
to indicate the typical state of the population.  Let 
$$
p=\frac{f_{AA}N_{AA}+f_{Aa}N_{Aa}/2}{f_{AA}N_{AA}+f_{Aa}N_{Aa}+f_{aa}N_{aa}}
$$
be  the  relative frequency of $A$ in the gametes. Then the population dynamics $t\mapsto (N^K_{AA,t}, N^K_{Aa,t}, N^K_{aa,t})$ is a birth and death process with three types and birth rates $b_{AA}, b_{Aa}, b_{aa}$ and death rates $d_{AA}, d_{Aa}, d_{aa}$ defined as follows.
\be
\label{naissance}
\begin{array}{rl}
b_{AA}=&(f_{AA}N_{AA}+{1\over2}f_{Aa}N_{Aa})\,p \\
 =& \frac{(f_{AA}N_{AA}+{1\over2}f_{Aa}N_{Aa})^2}{f_{AA}N_{AA}+f_{Aa}Y+f_{aa}N_{aa}}
,\\
b_{Aa}=&(f_{AA}N_{AA}+{1\over2}f_{Aa}N_{Aa})\,(1-p) +(f_{aa}N_{aa}+{1\over2}f_{Aa}N_{Aa})\,p \\
=& 2 {(f_{AA}N_{AA}+{1\over2}f_{Aa}N_{Aa})(f_{aa}N_{aa}+{1\over2}f_{Aa}N_{Aa})\over f_{AA}N_{AA}+f_{Aa}N_{Aa}+f_{aa}N_{aa}},\\
b_{aa}=&(f_{aa}N_{aa}+{1\over2}f_{Aa}N_{Aa})\,(1-p) \, \\
=& {(f_{aa}N_{aa}+{1\over2}f_{Aa}N_{Aa})^2\over f_{AA}N_{AA}+f_{Aa}N_{Aa}+f_{aa}N_{aa}}.\\
\end{array}
\ee
\be
\label{mort}
\begin{array}{rl}
d_{AA}=&\left(D_{AA}+{C_{AA,AA}\,N_{AA} + \,C_{AA,Aa} \,N_{Aa}\,+C_{AA,aa}\,N_{aa}\over K}
\right) N_{AA},\\
d_{Aa}=&\left(D_{Aa}\,\,+\, {C_{Aa,AA}\,N_{AA} + \,C_{Aa,Aa}\,N_{Aa}\,+C_{Aa,aa}\,N_{aa}\over K}\right) N_{Aa},\\
d_{aa}=&\left(D_{aa}\,\,\,\,+\, \,{C_{aa,AA}\,N_{AA} + \,C_{aa,Aa}\,N_{Aa}\,+C_{aa,aa}\,N_{aa}\over K}\right) N_{aa}.\\
\end{array}
\ee
To see this, it suffices to consider the generator \eqref{eq:generator-renormalized-IPS}
with $\mu_{K}=0$; for instance,  $K\, \langle \nu, f\rangle = f_{AA}N_{AA}+f_{Aa}N_{Aa}+f_{aa}N_{aa}$. 

\begin{prop} \label{largepoptri} Assume that the initial condition  $K^{-1}(N^K_{AA,0}, N^K_{Aa,0}, N^K_{aa,0})$ converges to a deterministic vector $(x_{0}, y_{0},z_{0})$ when $K$ goes to infinity. 
Then the normalized process  $K^{-1}(N^K_{AA,t}, N^K_{Aa,t}, N^K_{aa,t})$ converges in law when $K$ tends to infinity to the solution $(x(t),y(t),z(t))=\varphi_{t}(x_{0}, y_{0},z_{0})$ of
\be 
\label{sysdyn}
{d \over dt}\left(\begin{array}{c}
x(t)\\
y(t)\\
z(t)\\ 
\end{array}\right)=  X\big(x(t),y(t),z(t)\big)\;,
\ee
where 
\be
\label{leX}
 X(x,y,z)=\left(\begin{array}{c}
\tilde b_{AA}(x,y,z) - \tilde d_{AA}(x,y,z)\\
\tilde b_{Aa}(x,y,z) - \tilde d_{Aa}(x,y,z)\\
\tilde b_{aa}(x,y,z) - \tilde d_{aa}(x,y,z)\\
\end{array}\right)\;,
\ee
with
$$
\tilde b_{AA}(x,y,z)= \frac{(f_{AA}x+{1\over2}f_{Aa})(f_{AA}x+{1\over2}f_{Aa}y)}{
f_{AA}x+f_{Aa}y+f_{aa}z}\;,
$$
$$
\tilde d_{AA}(x,y,z)=
(D_{AA}+C_{AA,AA}\,x + \,C_{AA,Aa} \,y\,+C_{AA,aa}
\,z)\, x\;,
$$
and similar expressions for the other terms.
\end{prop}

Due to its special functional form, the vector field $ X$ has
some particular properties. We summarize some of  them in the
following Propositions.

\begin{prop}\label{stabfix}
The vector field \eqref{leX} has two fixed points 
$(\bar n_{AA},0,0)$ and $(0,0,\bar n_{aa})$ (denoted below by $AA$ and $aa$) where
$$
\bar n_{AA}=\frac{f_{AA}-D_{AA}}{C_{AA,AA}}\;,\qquad\mathrm{and}\qquad
\bar n_{aa}=\frac{f_{aa}-D_{aa}}{C_{aa,aa}}\;.
$$
The ($3\times 3$) Jacobian matrix $DX(AA)$ has the eigenvalues 
$-f_{AA}+D_{AA}$ (negative by assumption $(A2)$), 
$-C_{aa,AA}\;n_{AA}-D_{aa}<0$, and 
$$
S_{Aa,AA}=f_{Aa}-D_{Aa}-C_{Aa,AA}\;\bar n_{AA}.
$$
An analogous result holds for $DX(aa)$.

\end{prop}
This result follows from a direct computation left to the reader.

\bi
As we will see later on, the eigenvalue   $
S_{Aa,AA}$ will play a key role in the dynamics of trait substitutions. It describes the initial growth rate of the number of $Aa$ individuals in a resident population of $AA$ individuals and is  called the invasion  fitness of an $Aa$ mutant  in an $AA$ resident population. It  is a function of the allelic traits $u_{A}$ and $u_{a}$.

\bi
{\bf Notation}: When we wish to emphasize the dependence on the  two allelic traits $(u_{A},u_{a})$, we use the notation
\begin{equation}\label{fitness}
S_{Aa,AA} = S(u_{a};u_{A})= f(u_{A},u_{a}) -D(u_{A},u_{a})-C((u_{A},u_{a}),(u_{A},u_{A})) \frac{f(u_{A},u_{A}) -D(u_{A},u_{A})}{C((u_{A},u_{A}),(u_{A},u_{A}))}.
\end{equation}
Note that the function $S$ is not symmetric in $u_{A}$ and $u_{a}$ and that moreover 
\be
\label{fitness0}
S(u_{A}; u_{A}) = 0.
\ee

In Appendices \ref{casgen} and \ref{casduX}  the long term behavior of the flow generated by the vector field (\ref{leX}) is analyzed in more detail. The main conclusions are: 

\begin{prop}
First consider the case when the mutant and resident traits are precisely equal. Then the total population density goes to a unique equilibrium and the relative frequencies of the genotypes go to the Hardy-Weinberg proportions $(p^2,p(1-p,(1-p)^2)$, i.e., there exists  a globally attracting one-dimensional manifold filled with neutrally stable equilibria parametrized by $p$, with as stable manifolds the populations with the same $p$.

\noindent For the mutant and resident sufficiently close, this attracting manifold transforms into an invariant manifold connecting the pure resident and pure mutant equilibria. When  $S_{Aa,AA}>0$  the pure resident equilibrium attracts only in the line without any mutant alleles and its local unstable manifold is contained in the aforementioned invariant manifold (Theorem {\rm\ref{wsloc}}). When moreover the traits are sufficiently far from an evolutionarily singular point (defined by $\partial_1S(u_A;u_A)=0$)  the movement on the invariant manifold is from the pure resident to the pure mutant equilibrium, and any movement starting close enough to the invariant manifold will end up in the pure mutant equilibrium (Theorem {\rm \ref{tube}}).
\end{prop}


\section{The long term large population and rare mutations limit: trait substitution sequences}\label{longterm}

In this section we generalize the clonal theory of adaptive dynamics to the diploid case. We again make the combined large population and rare mutation assumptions, except that we now change the time scale to stay focused on the effect of the mutations. Recall that the mutation probability for an individual with genotype $(u_{1},u_{2})$ is  $\mu_{K}\in (0,1]$.
Thus the time scale of the mutations in the population is ${1\over K\,\mu_{K}}$. We study the long time behavior of the population process in this time scale and prove that it converges to a pure jump process stepping from one homozygote type to another. This process will be a generalization of the simple Trait Substitution Sequences (TSS) that  for the haploid case were heuristically derived in \cite{DL96}, and \cite{ M96} where they were called 'Adaptive Dynamics', and rigorously underpinned  in \cite{C06}, \cite{CM09}.

\medskip
Let us define the set of measures with single homozygote support.
\ben {\cal
    M}_0 = \bigg\{ \bar{n}_{AA} \delta_{(u_{A}, u_{A})}\ ;\
u_{A}\in {\cal U}  \hbox{  and }\ \bar{n}_{AA} \
\hbox{equilibrium of } \eqref{logistic}\bigg\}.\een

\me
We will denote by $J$ the subset of ${\cal U}$ where $\partial_1 S(u;u)$ vanishes.
We make the following hypothesis.

\begin{hypothesis}
For any $u\in J$ we have 
$$
\frac{d}{du}\partial_1 S(u;u)\neq 0\;.
$$
\end{hypothesis}
This hypothesis implies that the zeros of $\partial_1 S(u;u)$ are isolated (see
\cite{dieu}), and since ${\cal U}$ is closed and compact, $J$ is finite.


\begin{defi}
\label{ESS}
The points $u^*\in {\cal U}$ such that $ \partial_1 S(u^*;u^*)= 0$ are called  evolutionary singular strategies (ess). 
\end{defi}

Note that because of \eqref{fitness0}, 
$$\partial_2 S(u^*;u^*)=\partial_1S(u^*;u^*)=0.$$

Let us now define the TSS process which will appear in our asymptotic.
\begin{defi}\label{TSS}
For any $\sigma>0$, we define the pure jump process $(Z^\sigma_{t}, t\geq 0)$ with values in ${\cal U}$, as follows: its initial condition is $u_{A_{0}}$ and the process jumps from $u_{A}$ to $u_{a}=u_{A}+h$ with rate
\begin{equation}
\label{taux}
 f( u_{A},u_{A})\, \bar{n}_{AA}\,  \frac{[S(u_{A} +h; u_{A})]_+} 
    { f( u_{A},u_{A}+h)}\,  m_{\sigma}(u_{A},h)dh.
\end{equation}
   \end{defi}

\begin{rem}
Under our assumptions, the jump process $Z^\sigma$ is well defined on $\mathbb{R}_{+}$. Note moreover that the jump  from $u_{A}$ to $u_{a}$ only happens if the invasion fitness $S(u_{a}; u_{A})>0$. 
\end{rem}

We can now state our main theorem.

\begin{thm}
\label{maintheorem} 

 Assume \textup{(A)}. Assume that
  $\nu^{\sigma,K}_0= {\gamma_{K}\over K}\delta_{(u_{A_{0}}, u_{A_{0}})}$ with ${\gamma_{K}\over K}$ converging in law to $\bar{n}_{A_{0}A_{0}}$ uniformly bounded in $L^1$ and  such that $ \partial_1 S(u_{A_{0}},u_{A_{0}})\neq 0$. (That is, the initial population is monomorphic for a type that is not an ess).  Assume finally that
  \begin{equation}
    \label{eq:mu_K-K}
    \forall \,V>0,\quad {\ln K\over \sigma} \ll \frac{1}{K \mu_K}\ll \exp(VK), \quad \hbox{ as } K\to \infty.
  \end{equation}
 For $\eta>0$ introduce the stopping time 
 \be
\label{entrance}
T^{\sigma,K}_{\eta} = \inf\left\{ t>0, \frac{\langle \nu^{\sigma,K}_{t/K \mu_K}, d(., J) \rangle}{\langle \nu^{\sigma,K}_{t/K \mu_K}, 1 \rangle} \leq \eta\right\},
\ee
where $d$ is the  distance on the allelic trait space. 

Extend  ${\cal M}_F$ with the cemetery point $\partial$. 

 Then there exists $\sigma_{0}(\eta)>0$ such that for all
    $0<\sigma<\sigma_{0}(\eta)$,  the process $(\nu^{\sigma,K}_{t/K
    \mu_K}\mathbbm{1}_{\{T^{\sigma,K}_{\eta} \geq t\}} + \partial
    \mathbbm{1}_{\{T^{\sigma,K}_{\eta} < t\}}  ;t\geq 0)$ converges (in the sense of finite dimensional
  distributions on ${\cal M}_F$ equipped with the topology of the
  total variation norm) to the ${\cal 
    M}_0$-valued Markov pure jump process $(\Lambda^\sigma_t; t\geq 0)$ with 
    $$\Lambda^\sigma_t = \bar n( Z^\sigma_{t
   })\delta_{(Z^\sigma_{t }, Z^\sigma_{t })}\mathbbm{1}_{\{T^{\sigma}_{\eta} \geq t\}} + \partial
    \mathbbm{1}_{\{T^{\sigma}_{\eta} < t\}},$$
    where  
    $$T^{\sigma}_{\eta}  = \inf\left\{ t>0,  d(Z^\sigma_{t}, J)  \leq \eta\right\}.$$
The process  $(\Lambda^\sigma_t; t\geq 0)$ is defined as
  follows: $\Lambda^\sigma_0=\bar{n}_{A_{0}A_{0}}\delta_{(u_{A_{0}},
   u_{A_{0}})}$ and $\Lambda^\sigma$ jumps  
  $$
  \hbox{from }\
  \bar{n}_{A,A}\delta_{(u_{A}, u_{A})}\  \hbox{ to }\
   \bar{n}_{a,a}\delta_{(u_{a}, u_{a})}
  $$
  with $u_{a} = u_{A} +h$ and infinitesimal rate \rm{(\ref{taux})}.
\end{thm}

\begin{rem} \rm
Close to singular strategies the convergence to the TSS slows down. To arrive at a convergence proof it is therefore necessary to excise those close neighborhoods. This is done by means of the stopping times $T^{\sigma,K}_{\eta}$ and $T^{\sigma}_{\eta}$: we only consider the process for as long as it stays sufficiently far away from any singular strategies. Assumptions (A) imply that the thus stopped TSS $(Z^\sigma_{t})_{t}$ is well defined on $\mathbb{R}_{+}$. Since its jump measure is absolutely continuous with respect to the  Lebesgue measure, it follows  that 
$T^\sigma_{\eta}$ converges almost surely to $\infty$ when $\eta$ tends to $0$ (for any fixed $\sigma>0$). 
\end{rem}


\noindent We now roughly describe the successive steps of  the mutation, invasion and substitution dynamics making up the jump events of the limit process, following the biological heuristics of \cite{DL96,M96,Mip}. The details of the proof are described in Appendix \ref{prfmt}, based on the technical Appendices \ref{casgen}  and  \ref{casduX}.

\medskip
\noindent The time scale separation that underlies the limit in Theorem \ref{maintheorem} both simplifies the processes of  invasion and of the substitution of a new successful mutant on the population dynamical time scale and compresses it to a point event on the evolutionary time scale.   The two main simplifications of the processes of mutant invasion and substitution are the stabilization of the resident population before the occurrence of a mutation, simplifying the invasion dynamics, and the restriction of the substitution dynamics to a competition between two alleles. In the jumps on the evolutionary time scale $t/K \mu_K$ these steps occur in opposite order. First comes the attempt at invasion by a mutant, then, if successful, followed by its substitution, that is, the stabilization to a new monomorphic resident  population. After this comes again a waiting time till the next jump.

\noindent  To capture the stabilization of the resident population, we prove, on the assumption that  the starting population is monomorphic with genotype $AA$, that for arbitrary fixed $\varepsilon>0$  for large $K$  the population density $\langle\nu^{\sigma,K}_t,\mathbbm{1}_{\{(u_{A},u_{A})\}}\rangle$  with high probability stays in the $\varepsilon$-neighborhood of $\bar{{n}}_{AA}$  until the next allelic mutant $a$ appears. To this aim, we use large
deviation results for the exit problem from a domain (\cite{FW84}) already proved in \cite{C06} to deduce that  with high probability the time needed
for the population density to leave the $\varepsilon$-neighborhood of
$\bar{{n}}_{AA}$ is bigger than $\exp(VK)$ for some
$V>0$. Therefore, until this exit time, the
rate of mutation from  $AA$ in the population is close to
$\, \mu_K p_{AA}\,f_{AA} \,K\bar{{n}}_{AA}$ and thus, the first
mutation appears before this exit time if one assumes that
$$
\frac{1}{K \mu_K}\ll e^{VK}.
$$
Hence,  on the time scale $t/K \mu_K$ the population level mutation rate from  $AA$ parents is close to
$$  \,p_{AA}\, f_{AA} \,\bar{{n}}_{AA}. $$

\noindent To analyze the fate of these mutants $a$, we divide the population dynamics of the mutant alleles into the three phases shown in Fig.~\ref{substitution}, in a similar way as was done in \cite{C06}.

 \begin{figure}[!ht]
\begin{center}
\includegraphics[angle=-90,scale=.75]{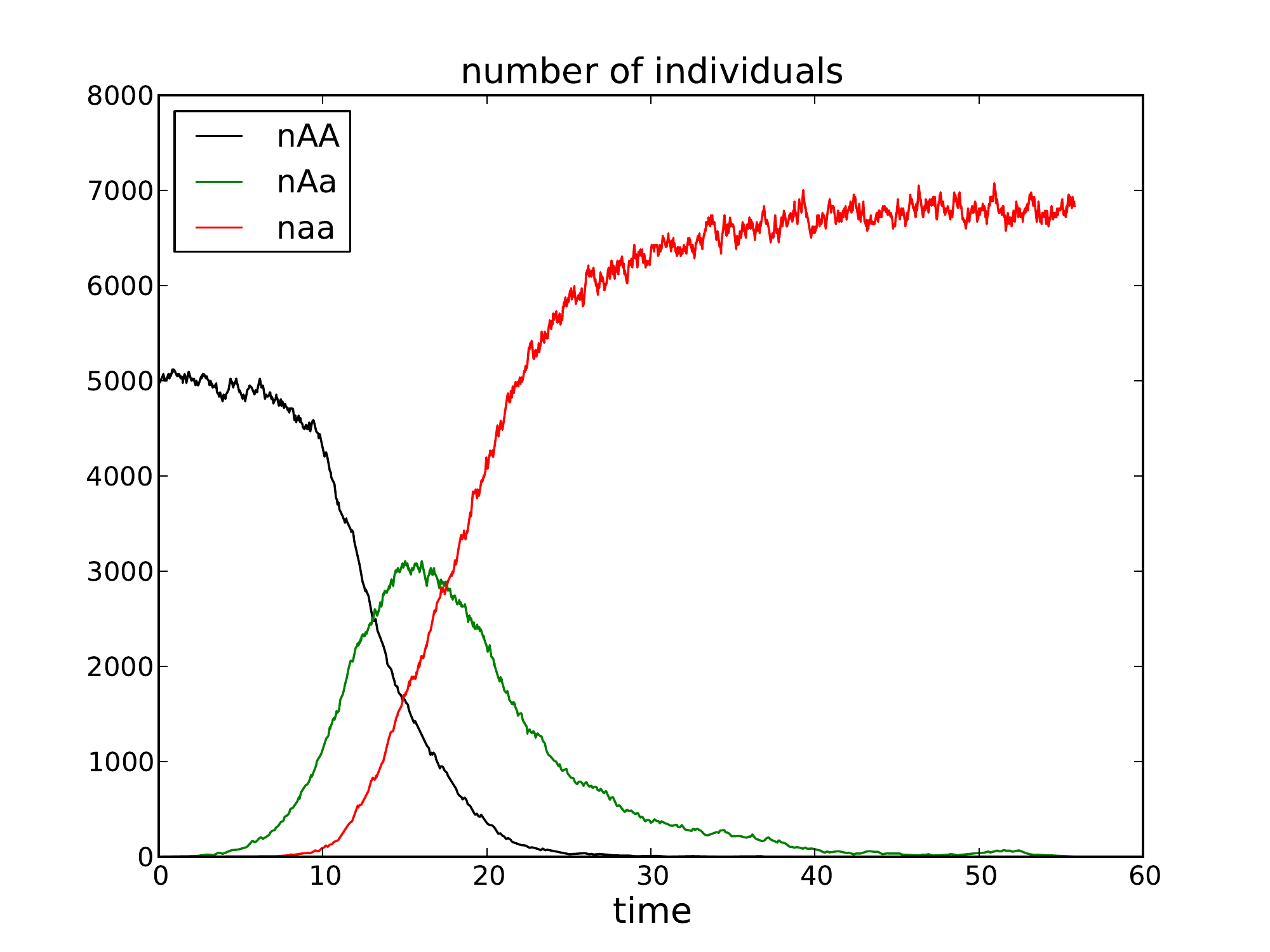}
\end{center}
\caption{Simulation of the three phases of mutant
 invasion.}\label{substitution}
\end{figure}

\noindent In the first phase (between time 0 and $t_1$ in
Fig.~\ref{substitution}), the number of mutant individuals of genotype $Aa$ or $aa$ is small,
and the resident population with genotype $AA$  stays close to its equilibrium density
$\bar{{n}}_{AA}$. Therefore, the dynamics of the mutant
individuals with genotypes $Aa$ and $aa$  is close to a bi-type birth and death  process with  birth rates
$f_{Aa} y + 2 f_{aa} z$ and $0$ and  death rates $(D_{Aa} +C_{Aa,AA}\bar{{n}}_{AA})\, y$ and $(D_{aa} +C_{aa,AA}\bar{{n}}_{AA})\, z$ for a state $(y,z)$.   If the fitness $S_{Aa; AA}$ is positive
(i.e. the branching process is super-critical), the probability
that the mutant population with genotype $Aa$ or $aa$ reaches  $K\,\varepsilon >0$ at some time
$t_1$ is close to the probability that the branching process reaches $K\,\varepsilon >0$, which is itself close to its survival probability
$\frac{[S_{Aa; AA}]_+}{f_{Aa}}$ when $K$ is large.


\noindent Assuming the mutant population with genotype $Aa$ or $aa$ reaches  $K\,\varepsilon >0$, a second phase starts.  When
$K\rightarrow+\infty$, the population densities
$(\langle\nu_t^{\sigma,K},\mathbbm{1}_{\{AA\}}\rangle,\langle\nu_t^{\sigma,K},\mathbbm{1}_{\{Aa\}}\rangle,\langle\nu_t^{\sigma,K},\mathbbm{1}_{\{aa\}}\rangle)$
are close to the solution of the dynamical  system
\eqref{sysdyn} with the same initial condition, on any
time interval $[0,T]$.  The study of this dynamical system (see Appendices \ref{casgen} and \ref{casduX}) implies  that, if the mutation step $u_{a} - u_{A}$ is
sufficiently small, then any solution to the dynamical system
starting in some neighborhood of
$(\bar{n}_{AA}, 0,0)$ converges to
the new equilibrium $(0, 0, \bar{n}_{aa})$ as time goes to
infinity.  Therefore,  with high probability the population densities reach the $\varepsilon$-neighborhood of $(0, 0, \bar{n}_{aa})$ at some time $t_2$.  Applying the results in Theorems \ref{wsloc} and \ref{tube} for the deterministic system to the approximated stochastic process, is justified by observing that the definition of the stopping times $T^{\sigma,K}_{\eta}$ and $T^{\sigma}_{\eta}$ implies that the allelic trait $u_{A}$ stays at all times away from the set $J$.

\noindent Finally, in the last phase, we use the same idea as in the
first phase: since $(0, 0, \bar{n}_{aa})$ is a strongly locally stable
equilibrium, we can approximate the densities of the
traits $AA$  and $Aa$  by a bi-type sub-critical branching process. Therefore, they reach $0$ in finite time and the process
comes back to where we started our argument (a monomorphic population),  until the next mutation.

\noindent In \cite{CM09}  it is proved  that the duration of these three phases is
of order ${\log K\over \sigma}$. Therefore, under the assumption
$$
\log K \ll \frac{\sigma}{K\mu_K},
$$
the next mutation occurs after these three phases with high probability.
Then the time scale Assumption \eqref{eq:mu_K-K} allows us to conclude, taking the limits $K$ tending to infinity and then $\varepsilon$ to $0$. Then we repeat the argument using the Markov property. 

\bi
\noindent Note that the convergence cannot hold for the usual Skorohod
topology and the space ${\cal M}_F$ equipped with the corresponding weak
topology. Indeed, it can be checked that the total mass of the limit
process is not continuous, which would be in contradiction with the
$C$-tightness of the sequence $(\nu^{\sigma,K}_{t/K \mu_K},t\geq 0)$, which would
hold in case of convergence in law for the Skorohod topology (since the jump amplitudes are equal to ${1\over K}$ and thus tend to $0$ as $K$ tends to infinity). 

\me
However, certain functionals of the process converge in a stronger sense. Let us for example consider the average over the population of the phenotypic trait $\phi$. This can be easily extended to more general symmetric functions of the allele. 

\begin{thm}
\label{SKO} Assume that $u \to \phi(u,u)$ is strictly monotone. Define 
$$T_{\phi, \eta}^{\sigma,K} = \inf\left\{ t>0, d\left( \frac{\langle \nu^{\sigma,K}_{t/K \mu_K}, \phi \rangle}{\langle \nu^{\sigma,K}_{t/K \mu_K}, 1 \rangle}, J_{\phi}\right) \leq \eta\right\},$$ 
where $J_{\phi} = \{\phi(u,u); u\in J\}$.

Under the assumptions of Theorem \ref{maintheorem}, the process $$(R^{\sigma,K}_{t}, t\geq 0)=\left(\frac{\langle \nu^{\sigma,K}_{t/K \mu_K}, \phi\rangle}{\langle \nu^{\sigma,K}_{t/K \mu_K}, 1 \rangle}\, \mathbbm{1}_{\{T^{\sigma,K}_{\phi,\eta} \geq t\}} , \, t\geq 0\right)$$ converges in law in the sense of the Skorohod  $M_{1}$ topology to the process $(\phi(Z^\sigma_{t}Z^\sigma_{t}) \mathbbm{1}_{\{T^{\sigma}_{\phi,\eta} \geq t\}}, \\  t \geq 0)$ where $T^{\sigma}_{\phi,\eta}  = \inf\left\{ t>0, d\left( \phi(Z^\sigma_{t}, Z^\sigma_{t}), J_{\phi}\right) \leq \eta\right\}$.
\end{thm}

The Skorohod  $M_{1}$  topology is a weaker topology than the usual $J_{1}$ topology, allowing processes with jumps tending to $0$ to converge to processes with jumps (see \cite{S56}). For a c\`ad-l\`ag function $x$ on $[0,T]$, the continuity modulus for the $M_{1}$ topology  is given by
$$
w_{\delta}(x) = \sup_{\stackrel{\scriptstyle 0\leq t_{1}\leq t\leq t_{2}\leq T;} {0\leq t_{2}-t_{1}\leq \delta}} d(x(t), [x(t_{1}), x(t_{2})]).
$$
Note that if the function $x$ is monotone, then $w_{\delta}(x) = 0$. 

\bi
\begin{proof}
From the results of Theorem \ref{maintheorem}, it follows easily that 
finite dimensional distributions of $(R^{\sigma,K}_{t}, t\geq 0)$ converge to those of $(\phi(Z^\sigma_{t}, Z^\sigma_{t}), t\geq 0)$. By  \cite{S56} Theorem 3.2.1, it remains to prove that for all $\eta>0$, 
$$\lim_{\delta\to 0} \limsup_{K\to \infty} \mathbb{P}(w_{\delta}(R^{\sigma,K}_{t}) >\eta) = 0.
$$

\noindent The rate of mutations of $(R^{\sigma,K}_{t}, t\leq T)$ being bounded, the probability that two mutations occur within a time less that $\delta$ is $o(\delta)$. It is therefore enough to study the case where there is  at most one mutation on the time interval $[0,\delta]$.   As  in the proof of Proposition \ref{largepoptri}, with probability tending to $1$ when $K$ tends to infinity,  the process $(R^{\sigma,K}_{t}, t\geq 0)$  is close to $F_{W_{\phi}}({t/K \mu_K})$ where $F_{W_{\phi}}$ is defined by
$$F_{W_{\phi}}(t) = \frac{\langle  \varphi_{t}(M_{0})\;,\; W_{\phi}\rangle}
{\langle \varphi_{t}(M_{0})\;,\; 1\rangle}\;,$$ and 
$$
 W_{\phi}=\left(\begin{array}{c}
\phi\big(u_{A},u_{A}\big)\\
\phi\big(u_{A},u_{a}\big)\\
\phi\big(u_{a},u_{a}\big)\\
\end{array}\right)
\;.
$$
Recall that $\varphi_{t}$ is the flow defined by  the  vector field
(see Proposition \ref{largepoptri}). 
Away from invading mutations, the function $F_{W_{\phi}}$ is constant and the modulus of continuity tends to $0$. Around an invading mutation, it follows from Corollary \ref{monotonie} that the function $F_{W_{\phi}}$ is monotone. Therefore the same conclusion holds. 

 \end{proof}

\section{Small mutational steps - the time scale  of the canonical equation}\label{CE}

\noindent We are now interested to study the convergence of the TSS when the mutation amplitude $\sigma$ tends to zero. Without rescaling time, the TSS trivially tends to a constant. In order to get a nontrivial limit, we have to rescale time adequately, namely with ${1\over \sigma^2}$, since $S(u_{A};u_{A}) = 0$.

\begin{thm} 
\label{canon} 
Assume that the initial values $Z^\sigma_{0}
$ are uniformly bounded in $L^2$ and that they converge to $Z^0_{0}$ as $\sigma$ tends to $0$. Then,  the sequence of processes $(Z^\sigma_{{t/\sigma^2}}, t\geq 0)$ tends in law in $\mathbb{D}([0,T], \mathbb{R})$ to the deterministic (continuous) solution $(u(t), t\geq 0)$ of the canonical equation
\be
\label{canon1}
{d\over dt}u(t) = \,f(u(t),u(t)) \, \bar n(u(t))\, \int_{\mathbb{R}} h\, [h\,\partial_1 S(u(t);u(t))]_{+} \, m(u(t),h)dh,
\ee
where 
$$  \bar n(u) = \frac{f(u,u) -D(u,u)}{C((u,u),(u,u))}.
$$
\end{thm}

\noindent The proof of this theorem is similar to the proof of Theorem 4.1 in  \cite{CM09}. 

In this general form the canonical equation is still of little practical use, although already some qualitative conclusions can be drawn from it. The trait increases whenever the fitness gradient $\partial_1 S(u;u)$ is positive and decreases when it is negative, i.e., movement is always uphill with respect to the current allelic fitness landscape $S(\cdot;u)$. The equilibria of  (\ref{canon1}) correspond to the allelic evolutionarily singular strategies, except that close to those strategies (\ref{canon1}) is no longer applicable since in their neighborhood the convergence of the underlying individual-based process to the simple TSS becomes slower and slower. So all we can deduce from the canonical equation (\ref{canon1}) is that for small mutational steps the trait substitution sequence will move to some close neighborhood of an allelic evolutionarily singular strategy. 

\begin{rem} \rm
If we had considered extended TSSes taking values in the powers of the trait space as is done in \cite{M96}, the convergence to the canonical equation would similarly have gone awry due to a slowing down of the convergence near evolutionarily singular strategies, and the occurrence of polymorphism close to some of them, with adaptive branching as a particularly salient  example; branching  can only be investigated with a time scaling different from the one for the canonical equation \cite{M96,CM09}.
\end{rem}

To get from the previous observation to some biological conclusion we need to decompose the genotypic fitness function $S$ into its ecological and developmental components
\begin{equation}
\begin{array}{rl}
\label{fitness2}
S_{Aa,AA} =& \tilde{S}(\phi_{Aa};\phi_{AA})= \tilde{f}(\phi_{Aa}) -\tilde{D}(\phi_{Aa})-\tilde{C}(\phi_{Aa},\phi_{AA}) \frac{\tilde{f}(\phi_{AA}) -\tilde{D}(\phi_{AA})}{\tilde{C}(\phi_{AA},\phi_{AA})}, \\
\phi_{Aa}=&\phi(u_A,u_a), \quad \phi_{AA}=\phi(u_A,u_A), \quad \tilde{f}(\phi_{Aa})=f(u_A,u_a),\quad \mathrm{etc.}
\end{array}
\end{equation}
and 
\begin{equation}
\partial_1 S(u;u)=\partial_1 \tilde{S}(\phi(u;u);\phi(u;u))\partial_1\phi(u,u).
\end{equation}
Hence, the allelic singular strategies are of two different types, ecological, characterized by $\tilde{S}(\phi(u;u);\phi(u;u))=0$, and developmental, characterized by $\partial_1\phi(u,u)=0$. On the phenotypic level the latter are perceived as developmental constraints (c.f. \cite{VD00}). 

To arrive at quantitative conclusions we have to make additional assumptions about the within individual processes. One often used assumption is that the mutation distribution is symmetric. With that assumption (\ref{canon1}) reduces to
\be
\label{canon2}
{d\over dt}u(t) =  \, {1 \over2}\,\bar n(u(t))\, V_\mathrm{a}(u(t))\partial_1 S(u(t);u(t)),
\ee
with $V_\mathrm{a}$ the allelic mutational variance. (The factor ${1\over2}$ comes from the fact that the integration is only over a half-line.) This equation can easily be lifted to the phenotypic level as
\be
\label{canon3}
{d\over dt}U(t) =  \,  \bar n(U(t))\, V_\mathrm{p}(U(t))\partial_1 \tilde{S}(U(t);U(t)),
\ee
with $U=\phi(u,u)$ and $V_\mathrm{p}$ the phenotypic mutational variance, an equation fully phrased in population level observables. The factor ${1\over2}$ is canceled by a factor 2 coming from the fact that the fitness $\tilde{S}$ refers to  heterozygotes with only one mutant allele, while after a substitution the other allele is also a mutant one. For this equation only the ecological singular strategies remain while developmental constraints appear in the form of $V_\mathrm{p}$ becoming zero (c.f. \cite{VD00}). (It is also possible to lift (\ref{canon1}) to the phenotypic level. However, the truncated first and second moments that appear in the resulting expression are no longer well-established statistics that can be measured independent of any knowledge of the surrounding ecology.)

\section{Discussion}
\label{disc}
This paper forms part of a series by a varied collection of authors that aim at putting the tools of adaptive dynamics on a rigorous footing \cite{MNG92,DL96,M96,GKMM98,CFM06,C06,DMM08,MT09,CM09,Mip,{KlebSagVatHacJag11}, bovier-champagnat-08} (see also \cite{diekmann-jabin-al-05,barles-perthame-06,CCP07,DJMG08}). It is the first  in the series  to treat the individual-based justification of the adaptive dynamics tools in a genetic setting. As such it forms the counterpart of the more heuristic, but also more general \cite{Mip}. We only consider unstructured Lotka-Volterra type populations and single locus genetics, in line with applied papers such as \cite{KG99,VD99,PP06,PB08}. For such models we proved the convergence (for large population sizes and suitably small mutation probabilities) of the individual-based stochastic process to the TSS of adaptive dynamics, and the subsequent convergence (for small mutational steps) of the TSS to the canonical equation. Not wholly unexpectedly, the results are in agreement with the assumed framework of the more applied work. Yet, to arrive at a rigorous proof new developments were needed, like the derivation  of a rigorous estimate for the probability of invasion in a dynamic diploid population (Appendix \ref{prfmt}), a rigorous, geometric singular perturbation theory based, invasion implies substitution theorem (appendix \ref{casduX}), and the use of the Skorohod $M_1$ topology to arrive at a functional convergence result for the TSS (Section \ref{longterm}).

Below we list the remaining biological limitations of the present results and the corresponding required further developments.

The first limitation is the assumption of an unstructured population.  For a a fair number of real populations the assumption of random deaths appears to match the observations, but no organisms reproduce in a Poisson process starting at birth. Moreover, in nature a good amount of population regulation occurs through processes affecting the birth rate, as when a scarcity of resources translates in  a delay of maturing to the reproductive condition. \cite{DMM08} heuristically treats very general life histories (although only for a finite number of birth states, a finite number of variables channeling the interaction between individuals, and a deterministic population dynamics converging to a unique equilibrium) based on the population dynamical modeling framework of \cite{DGMT98,DGHKMT01,DGM03}. However, it only considers the convergence to the canonical equation, starting from the TSS, conjectured to be derivable from the population dynamical model, with the goal of relating its coefficient functions to observationally accessible statistics of individual behavior. In fact, even the convergence to a deterministic population model, as in Theorem \ref{largepop}, does not easily fit in the scheme of \cite{FM04} in the (biologically common) cases where the movement of individuals through their state spaces depends directly or indirectly on the population size and composition.  (The special case where this movement  decomposes in a product of a population- and a state-dependent term is covered in \cite{Tran06,Tran08,FT09}).

A further limitation is that we assumed the trait to be governed by only a single locus (in keeping with a well-established tradition starting with \cite{MK65}). The more locus case still has to be worked out. The superficially more easy case with infinitely many loci, so that no mutant  ever occurs on the same locus, is considered from a heuristic perspective in \cite{{Mip,MetzKovel}}. However, the problem  of rigorously setting up the underlying  individual-based model as a limit for models with an ever increasing number of loci still needs to be tackled.

The final extension to be considered is to higher dimensional geno- and phenotypic trait spaces. We conclude with a heuristic discussion of the form such an extension will take. On the genotypic level the canonical equation will take essentially the same form as (\ref{canon1}) and (\ref{canon2}), with scalar $u$, $h$ and $\partial_1S$ replaced by vectors, and the mutational variance by a covariance matrix, just as this is written in \cite{DL96,CFM06,DMM08,CM09}  for the clonal and \cite{Mip,MetzKovel} for the Mendelian case. However, there is one remaining snag, which is the reason why we opted for treating only the one-dimensional case. In the directions orthogonal to the selection gradient the fitness landscape around the resident strategy has the same shape as at an evolutionarily singular strategy. In the one-dimensional case we opted for just removing the neighborhoods of the singular strategies. If we were to apply the same strategy for the higher dimensional case we would have to remove all residents. The way out is by observing that the directions where something awry may occur are but a very small minority among all possible directions in which mutations may occur. Heuristic calculations suggest that the trouble only occurs in a narrow double horn with a boundary that at the resident strategy is tangent to the linear manifold orthogonal to the selection gradient, so that when the mutational step size $\sigma$ goes to zero, the probability of a mutant ending up in that horn decreases as some higher power of $\sigma$. Moreover, in the directions orthogonal to the fitness gradient the fitness is a quadratic function, making the probability of invasion scale not linearly but quadratically with the size of any mutational steps in those directions. The main problem with such mutants  is that some of them may on the population dynamical time scale keep coexisting with the resident.  Further heuristic calculations then suggest that for such a resident pair the probability of invasion of a subsequent mutant more in the direction of the fitness gradient is to the lowest order of approximation - in the distance between the two residents - equal to the probability of invasion in a monomorphic population of the average type, and that such a mutant ousts both residents. Therefore the general (i.e., more type) TSS is close to a simple TSS in which those untoward mutants are just removed from the consideration, the smaller the mutational step the closer. We put rigorously underpinning this scenario forward as the last of our list of challenges.

\bi
\textbf{Acknowledgements}: This work benefitted from the support from the "Chair Mod\'elisation Math\'ematique et Biodiversit\'e of Veolia Environnement-Ecole Polytechnique-Museum National d'Histoire Naturelle-Fondation X" and from the ANR MANEGE.

\bigskip
\begin{flushleft} {\Large{ \textbf {Appendices}}} \end{flushleft}

\appendix

 \section{A few words about competition kernels}\label{compkern}
In the ecological literature the models described in Section \ref{sec:model} are known as Lotka-Volterra competition models  \cite{L25,V31}. The early LV models were all deterministic, phrased as ODEs corresponding to large population limits such as considered in the Section \ref{short term}, without mutations. The determinism together with the assumption of clonal reproduction obviated the need to separately model birth and deaths: competition was represented as its overall effect on the population growth rate. The later stochastic models, e.g. \cite{DL96,M96}, usually put the effect of competition only in the death rate, as otherwise the chosen linear form of the interaction might lead to negative birth rates.

The simplest case is when $C=0$. This is the case customarily put forward in population genetics textbooks as starting point for the derivation of their deterministic models for gene frequency change by selection,  but for the fact that population geneticists usually work in discrete time. The unnatural consequence that the population either will die out or will keep growing indefinitely is made invisible by transforming to relative frequencies. The more realistic case of non-selective competition, $C((u_{1}, u_{2});(v_{1}, v_{2}))=\overset{\lower0.5em\hbox{$\smash{\scriptscriptstyle\frown}$}}{C}(v_1, v_2)$, leads to the same population genetical equations. The selective pressures on the gene frequencies then do not change with the population size or composition as they are caused only by differences in the fixed mortalities and fertilities. 

Where in population genetics the early selection models assumed indefinitely growing populations, the early stochastic models, in continuous time the Moran-type models, assumed constant population sizes. Although later variable population sizes were introduced, it was just assumed that these sizes fluctuated between positive lower and upper bounds \cite{Karlin68, Donnelli85}.  Stochastic models with the population regulation represented in accordance with ecological tradition are relative newcomers (e.g. \cite{Redig}).

The  case where  the additional death rate incurred by an individual from its competitive interaction depends only on the genotype of the focal individual and not on that of  its competitors is known in the ecological literature as purely density dependent selection \cite{Rough71, Rough76, Rough79} , and in the mathematical literature as logistic population regulation. This logistic case can be generalized to $C((u_{1}, u_{2});(v_{1}, v_{2}))=\overset{\lower0.5em\hbox{$\smash{\scriptscriptstyle\smile}$}}{C}(u_1,u_2)\overset{\lower0.5em\hbox{$\smash{\scriptscriptstyle\frown}$}}{C}(v_1, v_2)$, when it is not the total density but e.g. the total biomass that determines the felt competitive effect and different genotypes have different biomasses.  A further generalization is that population growth is regulated by a finite number of variables, think for example of the combination of space and nitrogen depletion:
$$C((u_{1}, u_{2});(v_{1}, v_{2}))=\sum_{i=1}^k \overset{\lower0.5em\hbox{$\smash{\scriptscriptstyle\smile}$}}{C}_i(u_1,u_2)\overset{\lower0.5em\hbox{$\smash{\scriptscriptstyle\frown}$}}{C}_i(v_1, v_2).$$
The vector  $(\overset{\lower0.5em\hbox{$\smash{\scriptscriptstyle\frown}$}}{C}_1,\dots , \overset{\lower0.5em\hbox{$\smash{\scriptscriptstyle\frown}$}}{C}_k)^\mathrm{T}$ is known as the impact of the individuals on their environment, and the vector $(\overset{\lower0.5em\hbox{$\smash{\scriptscriptstyle\smile}$}}{C}_1,\dots , \overset{\lower0.5em\hbox{$\smash{\scriptscriptstyle\smile}$}}{C}_k)$ as their sensitivity a \cite{Mesz06}. The latter generalization is evolutionarily richer in that it can allow diversification, which is excluded by the earlier considered kernels.  In \cite{DMM08} it is shown heuristically that close to an evolutionarily singular strategy any clonal  model evolutionarily behaves like a Lotka-Volterra competition model of the above type with $k$ equal to one plus the dimension of the trait space. 

The above considerations all come from either ecology or population genetics, and originally were phrased for a fixed  finite number of  types, clonal ones in the ecological  and Mendelian ones in the population genetics literature. The first model characterizing these types in terms of traits was formulated by Robert MacArthur and Richard Levins \cite{MAL64}, see also \cite{MacA70}. This model was later used to great effect by a large number of authors  (e.g.  \cite{Levins68,  MAL67,May73, May74, Rough76,ChristFench77,Rough79, Slat80}, but see also \cite{Rough89}) to study species packing population dynamically as well as evolutionarily. The first genetic model of this type was studied by Freddy Bugge Christiansen and Volker Loeschcke \cite{ChrLoe80,LoeChr84,ChrLoe87} , who considered the possibilities for the coexistence of finite numbers of genotypes. Explicit trait-based LV-style birth and death process models with mutation only appeared on the scene with the birth of adaptive dynamics \cite{DL96, M96}.

The most common assumption in trait-based LV competition models \cite{MAL64,MacA70,MacA72,Rough79}  is that 
$$C((u_1,u_2);(v_1,v_2))=C((u_1,u_2);(u_1,u_2))
{\int Q(u_1,u_2)q((u_1,u_2);z) Q(v_1,v_2)q(v_1,v_2;z) \mathrm{d}z \over \int Q^2(u_1,u_2)q^2((u_1,u_2);z)\mathrm{d}z}.$$
Here $z\in \mathbb{R}$ is customarily interpreted as a trait of a fine-grained self-renewing resource with a fast logistic dynamics  that is supposed to be non-evolving. That is, it is assumed that a resource unit comprises close to infinitely many  very small particles, so that the resource dynamics can be treated as deterministic and that the turnover of the resource  is very fast so that it effectively tracks its deterministic equilibrium as set by the current consumer population. Functions of $(u_1,u_2)$ depend again on this argument through $\phi$. $Q$ is the average rate constant for the encounter and absorption of resource particles by our consumer individuals, expressed in resource units, while $q$ tells how this use is spread over the resource axis. 

\noindent The most commonly  used parametric form is 
$$ f(u_1,u_2)-D(u_1,u_2)=:r(u_1,u_2)=\bar{r}, $$ 
$$  
{{r(u_1,u_2)}\over{C(u_1,u_2);(u_1,u_2))}}=:k(u_1,u_2)=\exp \left(
-\frac 
{\left( \phi(u_1,u_2)-\phi_0 \right)^2}
{ 2 \sigma_k^2}
\right),  $$ 
$$Q(u_1,u_2)q((u_1,u_2);z)=\exp\left(-
\frac{\left(z-\phi(u_1,u_2)\right)^2}{\sigma_a^2}
\right),$$
leading to
$$C((u_1,u_2);(v_1,v_2))=\bar{r} \exp\left(-
\frac{\left(\phi(u_1,u_2)-\phi(v_1,v_2)\right)^2} {2\sigma_a^2} +
\frac{\left(\phi(u_1,u_2)-\phi_0\right)^2} {2 \sigma_k^2}
\right)  $$
Deterministic models based on this kernel have all sorts of nice mathematical properties, but Adaptive Dynamically they are a bit degenerate in that when $\sigma_a<\sigma_k$ the final stop for trait substitution sequences that result from the long term large population and rare mutations limit, as treated in Section \ref{longterm}, is a Gaussian distribution over trait space (c.f. \cite{Rough79}) whereas for almost any slightly different model the final stop has finite support \cite{GylMesz05, LeiDoeDieck08}. For this reason adaptive dynamics researchers started to use slightly modified expressions for $k$ or $C$. (When $K$ is still finite, the number of branches visible in simulations also stays finite, due to the early abortion of incipient ones, with the number of recognizable branches becoming larger with increasing $K$ and $\sigma_k/\sigma_a$ \cite{ClaesEA07,{ClaesEA08}}.) Exploring the consequences of all sorts of different competition kernels by now has become a little growth industry;  a good sample may be found in \cite{doe11}.

\begin{rem}\rm
The description of the mechanism underlying the competition kernel given above was  a bit brash, in keeping with biological tradition. Starting from an underlying fast logistic resource dynamics actually gives
$$f(u_1,u_2)=y(\phi(u_1,u_2))\left(\int{v(\phi(u_1,u_2),z)w(z)k_\mathrm{R}(z)\mathrm{d}z-d_1(\phi(u_1,u_2))}\right),  $$
$$D(u_1,u_2)=d_2(\phi(u_1,u_2))$$
$$
C((u_1,u_2);(v_1,v_2))=y(x)\int{v(\phi(u_1,u_2);z)\frac{w(z)k_\mathrm{R}(z)}{r_\mathrm{R}(z)}v(\phi(v_1,v_2);z)\mathrm{d}z}
$$
 and hence 
$$
Q(u_1,u_2)q((u_1,u_2);z)=Vv(\phi(u_1,u_2);z)\left(\frac
{w(z)k_\mathrm{R}(z)}
{r_\mathrm{R}(z)}
\right)^{{1}/{2}}
$$
with $y$ the yield, i.e., $y^{-1}$ is the resource mass needed to make one consumer, $w$ the mass of a resource unit, $v$ the rate constant of consumers encountering and eating resource units, $d_1$ the rate constant of consumer mass loss due to basal metabolism, and $d_2$ the  consumer mortality rate, $r_\mathrm{R}$ the low density reproductive rate of the resource, and $k_\mathrm{R}$ its carrying capacity. $V$ is some unknown proportionality constant. (In the above terms the time scale separation results from  both $r_\mathrm{R}$ and $v$ being very large and $y$ very small with the product of $y$ and $v$ being $\mathrm{O}(1)$.) Apparently the interpretation of $Q$ and $q$  is more complicated than the standardly attributed one based on the assumption of constant  $w k_\mathrm{R} / r_\mathrm{R}$.

\noindent Although time-honoured, the above described mechanistic underpinning  is not without flaws, as explicitly laid out by \cite{Chesson90}. In the derivation it is assumed that, but for the indirect coupling through the consumers, the dynamics of different resources are independent.  Even very similar resource populations do not compete. However, this is only possible if their ecological properties depend everywhere discontinuously on the trait $z$, since the assumed logistic nature of the resource dynamics means that there is non-negligible  competition between equal resource particles. The alternative assumption alluded to by MacArthur \cite{MacA72} that the intrinsic resource dynamics is of a chemostat type (as can be approximately the case for seeds from perennial plants) also is problematical: Under the reasonable assumption that the resource mass removed by a consumer population equals the mass this population acquires, the detrimental effect from competition  becomes non-linear in the competitor densities, instead of being  simply representable by a competition kernel. 
 \end{rem}

\section{Properties of the vector field (\ref{leX}.)} 
\label{casgen}

\subsection{Neutral case.}
We first consider the case of neutrality between the $A$ and $a$
alleles, namely  $f_{a_{1}a_{2}}=f$, $D_{a_{1}a_{2}}=D^{0}$ and 
$C_{a_{1}a_{2}\,,\,b_{1}b_{2}}=D^{1}$ for 
$a_{1}a_{2},b_{1}\,b_{2}=AA,Aa,aa$. 
We have in this case with $n=x+y+z$
$$
p=\frac{x+y/2}{n}
$$ 
which is the proportion of allele $A$.
We get for the vector field
$$
 X_{0}=\left(\begin{array}{c}
f(x+y/2)p-(D^{0}+D^{1}n)x\\
f(x+y/2)(1-p)+f(z+y/2)p-(D^{0}+D^{1}n)y\\
f(z+y/2)(1-p)-(D^{0}+D^{1}n)z\\
\end{array}\right)
$$

\begin{thm}\label{neutre} 
The vector field $ X_{0}$ has a line of fixed points given by 
$$
\Gamma_{0}(v)=
\begin{pmatrix}{\frac{\dst v^2-2\,{ n_0}\,v+{ n_0}^2}{\dst 4\,{ n_0}}}\cr 
 -{\frac{\dst v^2-{ n_0}^2}{\dst 2\,{ n_0}}}\cr \frac{\dst {v^2+2\,{ n_0}\,v+
 { n_0}^2}}{\dst 4\,{ n_0}}\cr 
\end{pmatrix}\;,
$$
with $n_{0}=(f-D^{0})/D^{1}$. 
That is, we have for any $v$, $ X_{0}(\Gamma_{0}(v))=0$. 
The parametrization with $v$ is chosen such that the 
differential of the vector field $ X_{0}$ at each point of  the curve
$\Gamma_{0}$,
$DX_{0}(\Gamma_{0}(v))$,  has the
three eigenvectors
$$
 e_{1}(v)=\Gamma_{0}(v))=\begin{pmatrix}
\frac{\dst v^2-2\,{ n_0}\,v+{
      n_0}^2}{\dst 4\,{ n_0}}\cr  
 -\frac{\dst v^2-{ n_0}^2}{\dst 2\,{ n_0}}\cr \frac{\dst v^2+2\,{ n_0}\,v+
 {{ n_0}^2}}{\dst 4\,{ n_0}}\cr 
\end{pmatrix}\;,
$$
$$
 e_{2}(v)=\frac{\dst d\Gamma_{0}(v)}{\dst dv}=
\begin{pmatrix}
\frac{\dst v- n_0}{\dst 2\, n_0}\cr 
 -\frac{\dst v}{\dst  n_0}\cr 
\frac{\dst v+ n_0}{\dst 2\, n_0}\cr 
\end{pmatrix}
\;, 
$$
$$ e_{3}(v)=\frac{\dst d^{2}\Gamma_{0}}{\dst dv^{2}}=
\frac{1}{2n_{0}}\begin{pmatrix}
1\\
-2\\
1\\
\end{pmatrix}
$$
with respective eigenvalues $D^{0}-f<0$, $0$, and $-f<0$. The corresponding eigenvectors of the
transposed matrix $DX_{0}(\Gamma_{0}(v))^{t}$, to be denoted by by $ \beta_{1}(v)$, $ \beta_{2}(v)$ and $\beta_{3}(v)$ can be normalized such that for
any $i,j,\in\{1,2,3\}$ and any $v$
$$
\langle  \beta_{i}(v)\,,\,  e_{j}(v)\rangle=\delta_{i,j}\;.
$$
\end{thm}

\begin{proof}
This is easily seen by using the standard variables: total population density, $n=x+y+z$,
relative frequency of the $A$ allele, $p=(x+y/2)/n$, and excess heterozygosity realtive to the Hardy-Weinberg proportion,  $h=y/n-2p(1-p)$.

\noindent In these new coordinates, the vector field $ X_{0}$ becomes the
vector field $ Y_{0}$ given by
$$
 Y_{0}(n,p,h)=\left(\begin{array}{c}
f-(D^{0}+D^{1}\,n)\,n\\
0\\
-f\,h\\
\end{array}\right)\;.
$$
This vector field obviously vanishes on the line $n=n_{0},\;
h=0$. One gets immediately the results by taking
$v=n_{0}\,(1-2\,p)$. The spectral results follow by standard
computations. 
\end{proof}

\subsection{Small perturbations.}

We now assume that mutations are small. We denote by $\mut$ the variation of the
allelic trait $\mut=u_{a}-u_{A}$. The vector field depends on $\mut$
and will be denoted by $ X(\mut,M)$. We assume regularity in $\mut$
and $M$, and observe that $ X(0,M)= X_{0}(M)$.

In practice we will apply our results to the vector field \eqref{leX}
which has a particular algebraic form. It is however convenient to
derive the perturbation results in full generality. We will come back
to the particular case of \eqref{leX} in section \ref{casduX}.

From now on, we will assume that the vector field $ X(\mut,\,\cdot\,)$
 satisfies the following properties for any $x$, any $z$ and any $\mut$
\begin{equation}\label{axesinv}
X_{x}\big(\mut,(0,0,z)\big)=X_{y}\big(\mut,(0,0,z)\big)=0\;,
$$
and
$$
X_{z}\big(\mut,(x,0,0)\big)=X_{y}\big(\mut,(x,0,0)\big)=0\;.
\end{equation}
This comes from the fact that pure homozygotic populations stay pure
homozygotic forever.

Our goal in this section is to understand the time asymptotic of the
flow associated to the vector field $ X(\mut,M)$.

Since the curve $\Gamma_{0}$ is transversally
hyperbolic (even transversally contracting, see Proposition 
\ref{neutre}) for the vector field $ X_{0}$,
we can apply Theorem 4.1 in \cite{HPS} to conclude that for
$\mut$ small enough, there is an attracting  curve $\Gamma_{\mut}$ invariant
by $ X$.  Moreover, $\Gamma_{\mut}$ is regular and converges to
$\Gamma_{0}$ when  $\mut$ tends to zero.  In other words, there is a
small enough tubular neighborhood $\mathscr{V}$ of $\Gamma_{0}$ such
that for any $|\mut|$ small enough, $\Gamma_{\mut}$ is contained in 
$\mathscr{V}$ and attracts all the orbits with initial conditions in 
$\mathscr{V}$. (For earlier, weaker results in this direction for general differential and difference equation population dynamical models without genetics see \cite[Appendix B]{GeritzTube,DercoleRinaldi08}.)

Applying Theorem 4.1 in \cite{HPS} requires
that the curve $\Gamma_{0}$ is a compact manifold without boundary,
but this is not the case here. However one can
perform some standard surgery to put our problem in this form in a
neighborhood of the part of $\Gamma_{0}$ which lies in the positive
quadrant which is the only part of phase space that matters for us.

\subsubsection{Location of the zeros of the perturbed vector field.}

Since the curve $\Gamma_{\mut}$ is invariant and (locally) attracting
for the flow associated to the vector field $ X(\mut,M)$, where $M$ stands for the vector $x,y,z)$ it is enough
to study the flow on this curve. In particular, since $\Gamma_{\mut}$
is a curve, if the vector field does not vanish on $\Gamma_{\mut}$
except at the intersections with the lines $x=y=0$ and $y=z=0$ (the
fixed points $aa$ and $AA$ respectively see Theorem \ref{stabfix}), we
know that the orbit of any initial condition on $\Gamma_{\mut}$
(between $AA$ and $aa$)  will converge either to $AA$ or to $aa$.

We now look for the fixed points on $\Gamma_{\mut}$ of the flow
associated to the vector field $ X(\mut,M)$ which are the points where
the vector field vanishes. Since $\Gamma_{\mut}$ is attracting, it is
equivalent (and more convenient) to look for the fixed points in
$\mathscr{V}$.

It is convenient to use for this study 
local  frames in  the  tubular neighborhood $\mathscr{V}$ of
$\Gamma_{0}$.  There are many possibilities for defining such frames,
 we found that a convenient one is to represent a point
$M$ by the parametrisation
$$
M(v,r,s)=\Gamma_{0}(v)+r  e_{1}(v)+s  e_{3}(v) 
= (1+r)\Gamma_{0}(v)+s\frac{d^{2}\Gamma_{0}(v)}{dv^{2}}\;.
$$
with $v\in[-n_{0}-\delta, n_{0}+\delta]$, $r\in[-\delta,\delta]$,
$s\in[-\delta,\delta]$ with $\delta>0$ to be chosen small enough later
on. We observe that $M(v,0,0)=\Gamma_{0}(v)$.

The Jacobian of the transformation $(v,r,s)\mapsto
(x,y,z)=M(v,r,s)$ is equal to
$
- (1+r)/2
$
and therefore does not vanish if $0<\delta<1$. It is easy to
verify that if $\delta>0$ is small enough, 
the map $(v,r,s)\mapsto M(v,r,s)$ is a diffeomorphism of
$[-n_{0}-\delta, n_{0}+\delta]\times [-\delta,\delta]^{2}$ to a close
neighborhood of $\mathscr{V}$ (provided this tubular neighborhood is
small enough).  In particular, once $\delta>0$ is chosen, for any
$\mut>0$ small enough, $\mathscr{V}$ contains the intersection of
$\Gamma_{\mut}$ with the first quadrant (by continuity of
$\Gamma_{\mut}$ in $\mut$).

In order to find the zeros of the vector field $ X(\mut,M)$, we will
use convenient linear combinations of its components which reflect the
fact that the flow is transversally hyperbolic. We will first equate
to zero two linear combinations of the components, and by the implicit
function theorem this will lead to a curve containing all possible
zeros. We will then look at the points on this curve where the third
(independent) linear combination of the components vanishes.

\begin{prop}\label{courbezeros}
For any $\delta>0$ small enough, there is a number
$\mut_{0}=\mut_{0}(\delta)$
such that for any $\mut\in [-\mut_{0},\mut_{0}]$ 
there is a smooth curve $\mathscr{Z}_{\mut}
=(r_{\mut}(v),s_{\mut}(v))\subset\RR^{2}$,
 depending smoothly on $\mut$,
and converging to $0$ when $\mut$ tends to zero
such that  for any $v\in[-n_{0}-\delta,n_{0}+\delta]$ we have
$$
\langle  \beta_{1}(v)\,,\, 
X\big(\mut,M(v,r_{\mut}(v),s_{\mut}(v))\big)\rangle=
\langle  \beta_{3}(v)\,,\,
 X\big(\mut,M(v,r_{\mut}(v),s_{\mut}(v))\big)\rangle=0\;. 
$$
Moreover, if a point $(v,r,s)$ with $v\in
[-n_{0}-\delta,n_{0}+\delta]$, $r$ and $s$ small enough is such that
$$
\langle  \beta_{1}(v)\,,\, X\big(\mut,M(v,r,s)\big)\rangle =
\langle  \beta_{3}(v)\,,\, X\big(\mut,M(v,r,s)\big)\rangle =0
$$
then $(r,s)=(r_{\mut}(v),s_{\mut}(v))$.
\end{prop}
\begin{proof}
Consider the map $F$ from $\RR^{2}\times \RR^{2}$ to $\RR^{2}$ given by
$$
F\big((\mut,v),(r,s)\big)=
\big(\langle  \beta_{1}(v)\,,\, X\big(\mut,M(v,r,s)\big)\rangle,
\langle  \beta_{3}(v)\,,\, X\big(\mut,M(v,r,s)\big)\rangle
\big)\;.
$$
For any  $v_{0}\in [-n_{0}-\delta,n_{0}+\delta]$, and $|\mut|$ small
enough, the differential of
$F$ in
$(r,s)$ at $(0,v_{0},0,0)$ is invertible. This follows by continuity 
from the same
result in $\mut=0$ where the determinant of the differential is $f(f-D^{0})$.
Therefore, by the implicit
function theorem (see for example \cite{dieu}), for any $v_{0}\in
[-n_{0}-\delta,n_{0}+\delta]$, there is an open 
neighborhood $U_{v_{0}}$ of $(v_{0},0)$ in $\RR^{2}$ and two regular
functions functions on  $U_{v_{0}}$,  $r^{v_{0}}$ and $s^{v_{0}}$ such
that for any $(\mut,v)\in U_{v_{0}}$ we have
$$
F\big((\mut,v),(r^{v_{0}}(\mut,v),s^{v_{0}}(\mut,v))\big)=0\;.
$$
Since the set $[-n_{0}-\delta,n_{0}+\delta]\times \{0\}$ is compact in
$\RR^{2}$, we can find a finite sequence $v_{1},\ldots,v_{m}$ such
that the finite sequence of sets $(U_{v_{j}})$ is a finite open
cover of $[-n_{0}-\delta,n_{0}+\delta]\times \{0\}$. We now define the
functions $r$ and $s$ in the tubular neighborhood $\cup_{j}U_{v_{j}}$
of $[-n_{0}-\delta,n_{0}+\delta]\times \{0\}$ by
$$
r(\mut,v)=r^{v_{j}}(\mut,v)\;,\quad s(\mut,v)=s^{v_{j}}(\mut,v)\;,
\quad \mathrm{for} (\mut,v)\in U_{v_{j}}\;.
$$
This definition is consistent since if $(\mut,v)\in U_{v_{j}}\cap 
U_{v_{\ell}}$ with $\ell\neq j$
we have
$r^{v_{j}}(\mut,v)=r^{v_{\ell}}(\mut,v)$ and
$s^{v_{j}}(\mut,v)=s^{v_{\ell}}(\mut,v)$ by the uniqueness of the 
solution in the implicit function theorem. The last assertion of the
proposition follows also from the uniqueness of the 
solution in the implicit function theorem.
\end{proof}

It follows immediately from the above result that 
 the vector field $ X(\mut,\,\cdot\,)$ vanishes in
a small enough neighborhood of $\Gamma_{0}$ if and only if
$$
\langle  \beta_{2}(v)\,,\,
X\big(\mut,M(v,r_{\mut}(v),s_{\mut}(v))\big)\rangle
=0\;,
$$
which at a given $\mut$ is an equation for $v$. 

We analyze a neighborhood of the point $\mut=0$. We first observe
that
$$
\langle  \beta_{2}(v)\,,\, X\big(0,M(v,r_{0}(v),s_{0}(v))\big)\rangle=
\langle  \beta_{2}(v)\,,\, X\big(0,M(v,0,0)\big)\rangle
$$
$$
=
\langle  \beta_{2}(v)\,,\, X\big(0,\Gamma_{0}(v))\big)\rangle=0\;.
$$
Therefore by the Malgrange preparation Theorem \cite{gg} (the
Weierstrass preparation Theorem in the analytic setting), we can write
\begin{equation}\label{malgrange}
\langle  \beta_{2}(v)\,,\,
X\big(\mut,M(v,r_{\mut}(v),s_{\mut}(v))\big)\rangle= 
\mut^{2}h(\mut,v)+\mut g(v)\;.
\end{equation}
\begin{lem}
The function $g$ in \eqref{malgrange} is given by
$$
g(v)=\langle  \beta_{2}(v)\,,\partial_{\mut} X(0,\Gamma_{0}(v))\rangle\;.
$$
\end{lem}
\begin{proof}
We have
$$
g(v)=\bigg\langle  \beta_{2}(v)\,,\,\big(\partial_{\mut} 
X\big(\mut,M(v,r_{\mut}(v),s_{\mut}(v))\big)
_{\big|\mut=0}\bigg\rangle
$$
$$
=\langle  \beta_{2}(v)\,,\,\partial_{\mut}
X\big(0,\Gamma_{0}(v)\big)\rangle
+\bigg\langle  \beta_{2}(v)\,,\,D
X(0,\Gamma_{0}(v))\;\partial_{r}M(v,0,0)\; 
\partial_{\mut}r_{\mut}(v))_{\big|\mut=0}\bigg\rangle
$$
$$
+\bigg \langle  \beta_{2}(v)\,,\,D
X(0,\Gamma_{0}(v))\;\partial_{s}M(v,0,0)\; 
\partial_{\mut}s_{\mut}(v))_{\big|\mut=0}\bigg\rangle
$$
$$
= \langle  \beta_{2}(v)\,,\, \partial_{\mut}
X\big(0,\Gamma_{0}(v)\big)
+\bigg \langle  \beta_{2}(v)\,,\,D X(0,\Gamma_{0}(v))\; e_{1}(v)\;
\partial_{\mut}r_{\mut}(v))_{\big|\mut=0}\bigg\rangle
$$
$$
+\bigg \langle  \beta_{2}(v)\,,\,D X(0,\Gamma_{0}(v))\; e_{3}(v)\;
\partial_{\mut}s_{\mut}(v))_{\big|\mut=0}\bigg\rangle\;.
$$
The lemma follows at once from Proposition \ref{neutre}.
\end{proof}

The following result gives conditions for the
perturbed vector field to have only two fixed points near $\Gamma_{0}$.

\begin{thm}\label{leve}
Assume the function 
$$
g(v)=\langle  \beta_{2}(v)\,,\partial_{\mut} X(0,\Gamma_{0}(v))\rangle\;.
$$
satisfies $dg/dv(\pm n_{0})\neq0$ and does not vanish in
$(-n_{0},n_{0})$.  Then for $|\mut|$
small enough (but non zero), the vector field $ X$ has only two zeros
in a  tubular neighborhood of $\Gamma_{0}$. These zeros are 
$(n_{AA}(\mut),0,0))$ and $(0,0,n_{aa}(\mut))$ with $n_{AA}(\mut)$ and 
$n_{aa}(\mut)$ regular near $\mut=0$ and 
$n_{AA}(0)=n_{aa}(0)=n_{0}$.
\end{thm}
As we will see in the proof $g(\pm n_{0})=0$ and the condition
$dg/dv(\pm n_{0})\neq0$ ensures that these zeros are isolated. 

\begin{proof}
We observe that 
$$
 X\big(\mut,\Gamma_{0}(-n_{0})\big)= X\big(\mut,n_{0},0,0)\big)\;,
$$
hence
$$
\partial_{\mut}X_{y}\big(\mut,\Gamma_{0}(-n_{0})\big)=
\partial_{\mut}X_{z}\big(\mut,\Gamma_{0}(-n_{0})\big)=0\;.
$$
On the other hand, by a direct computation one gets
$$
 \beta_{2}(-n_{0})=\left(\begin{array}{c}
0\\1\\ 2
\end{array}\right)
$$
and we get $g(-n_{0})=0$\;. Similarly one has $g(n_{0})=0$.

Since the functions $g$ and $h$ in 
\eqref{malgrange} are regular, for
$|\mut|$ small, it follows that the function 
$v\to \langle  \beta_{2}(v)\,, 
X\big(\mut,M(v,r_{\mut}(v),s_{\mut}(v))\big)$ can
vanish only in neighborhoods of points where $g$ vanishes.
We conclude that if $g$ does not vanish on the open interval
$]-n_{0},n_{0}[$, and 
$$
\frac{dg}{dv}(\pm n_{0})\neq 0\;,
$$ 
there is a number $\delta'>0$ such that 
for $|\mut|$ small enough non zero, the function
$v\to \langle  \beta_{2}(v)\,, 
X\big(\mut,M(v,r_{\mut}(v),s_{\mut}(v))\big)$ has
at most two zeros in the interval $[-n_{0}-\delta',n_{0}+\delta']$. Such
zeros must be  simple and near $\pm n_{0}$. By Theorem \ref{stabfix}
we conclude that these two zeros exist and are 
the two fixed points $aa$ and $AA$ respectively.
\end{proof}

\section{Applications to the process of mutant substitution}
\label{casduX}

Recall that in our setting, the resident population is monomorphic with genotype $(u_{A},u_{A})$. The mutant allelic trait $u_{a}$ is given by
$$u_{a} = u_{A} + \mut,$$
where $\mut$ has been chosen according to the distribution $m_{\sigma}(u_{A}, h) dh$ and therefore $|\mut|\leq \sigma$.
 
\subsection{The stable manifold of the AA fixed point.}

As we have seen before in Theorem \ref{stabfix} 
the stability of the fixed point $AA$
can be decided by looking at the fitness of the mutant.
 We will need later on a property of the stable
manifold in the case where this fixed point is unstable.

\begin{thm}
\label{wsloc}
For $|\mut|$ small enough, if $S_{Aa,AA}(\mut)>0$, the local stable manifold 
of the unstable fixed point $AA$  intersects the closed positive 
quadrant only along the line $y=z=0$. The local  unstable manifold is 
contained in the curve $\Gamma_{\mut}$.
\end{thm}
\begin{proof}
Hyperbolicity follows from Theorem \ref{stabfix},
 and we can apply Theorem 5.1 in \cite{HPS}. 
From Theorem \ref{stabfix}, one finds that
 the Jacobian matrix  $D X_{AA}$ has three eigenvectors
$$
 E_{1}(\mut)= e_{1}(-n_{0})+\mathcal{O}(\mut)\;,\quad
 E_{2}(\mut)= e_{2}(-n_{0})+\mathcal{O}(\mut)\;,\quad
 E_{3}(\mut)= e_{3}+\mathcal{O}(\mut)\;,\quad\;,
$$
with respective eigenvalues $D^{0}-f+\mathcal{O}(\mut)$,
$\mathcal{O}(\mut)$, $-f+\mathcal{O}(\mut)$. 

It follows from Theorem 5.1 in \cite{HPS} that
the local stable manifold
$W^{s,\,loc}_{AA}$ of $AA$ is a piece of regular manifold tangent in $AA$ to
the two dimensional affine stable  subspace $E^{s}_{AA}(\mut$ with origin
in $AA$, and spanned by the vectors $ E_{1}(\mut)$ and $ E_{3}(\mut)$.

The $x$ axis ($y=z=0$) is invariant by the vector field and is
contained in the stable manifold. The first result follows from the fact
that $E^{s}_{AA}(\mut)$ intersects the closed positive 
quadrant only along the line $y=z=0$.

Since the local (one dimensional) unstable manifold 
$W^{u,\,loc}_{AA}(\mut)$ of $AA$
is tangent to the
linear unstable direction in $ E_{2}(\mut)$ in $AA$, 
it is enough to show that this
direction points inside the quadrant. This follows immediately from 
the expression of $ E_{2}(\mut)$. By uniqueness of the invariant curve
(see Theorem 5.1 in \cite{HPS}), we conclude that
$W^{u,\,loc}_{AA}(\mut)\subset \Gamma_{\mut}$, and the result follows by the
invariance of the positive quadrant by the flow.
\end{proof}

\subsection{Invasion and fixation conditions.}\label{mandel}

Recall that the 
 functions $f(u_{1},u_{2})$, $D(u_{1},u_{2})$ and 
$C\big((u_{1},u_{2}),(v_{1},v_{2})\big)$  are symmetric in
$(u_{1},u_{2})$ and $(v_{1},v_{2})$.  Since $u_{a}=u_{A}+\mut$, we have 
$$
f_{AA}=f(u_{A},u_{A})\;, \qquad f_{Aa}=f(u_{A}+\mut,u_{A})\;, \qquad
\mathrm{etc.},
$$
 and
$$
f_{Aa}=f_{AA}+\frac{1}{2}\frac{d
  f_{AA}}{du}\mut+\mathcal{O}(\mut^{2})\;, \quad 
f_{aa}=f_{AA}+\frac{d
  f_{AA}}{du}\mut+\mathcal{O}(\mut^{2})\;, \qquad
\mathrm{etc.}
$$

\me
After some elementary computations one gets
$$
g(v)=-\frac{1}{2 \;n_{AA}}
\;\frac{ dS_{Aa,AA}}{d\mut}(0)\;(v^{2}-n_{AA}^{2})\;.
$$
Therefore, if 
$$
\frac{ dS_{Aa,AA}}{d\mut}(0)
\neq 0
$$  
the function $g$ vanishes only for $v=\pm n_{AA}$, and 
the  vector field $ X(\mut,\,.\,)$ has for small $|\mut|\neq0$
only two fixed points near the intersection of the curve $\Gamma_{0}$ with
the positive quadrant (these fixed points are on the lines $x=y=0$ and
$z=y=0$). 

Note that at neutrality we have $S_{Aa,AA}(0)=0=S_{Aa,aa}(0)$, hence
$$
S_{Aa,AA}(\mut)=\frac{dS_{Aa,AA}}{d\mut}(0)\mut+\mathcal{O}(\mut^{2})\;,
$$
ans similarly for $S_{Aa,aa}(\mut)$.

Hence, if $\frac{dS_{Aa,AA}}{d\mut}(0)\neq 0$, for $|\mut|$ small
enough, the stability of
$AA$ is determined by the sign of
$\frac{dS_{Aa,AA}}{d\mut}(0)\mut$ (and similarly for $aa$).

By a direct computation, one gets
$$
\frac{dS_{Aa,AA}}{d\mut}(0)=-\frac{dS_{Aa,aa}}{d\mut}(0)\;.
$$
Hence the two fixed points have opposite stability, therefore if
invasion occurs it implies fixation. 
The fixed point $AA$ is stable (the mutant does not invade) if
$\mut$ and  $dS_{Aa,AA}/d\mut(0)$
have  opposite sign. 

We now summarize  these results. We denote by $\Gamma_{\mut}^+$ the
piece of $\Gamma_{\mut}$ contained in the positive quadrant.  

\begin{thm}
\label{tube}
For $\mut$ non zero of small enough modulus, 
if $\,\mut \,dS_{Aa,AA}/d\mut(0)>0$ (which implies
$\, dS_{Aa,AA}/d\mut(0)\neq 0$)
the fixed point  $AA$ is unstable and we have fixation for the
macroscopic dynamics.

More precisely, the curve $\Gamma_{\mut}^+$ is the piece of unstable manifold between $AA$ and $aa$.  There exists an invariant tubular neighborhood $\mathscr{V}$ of $\Gamma_{\mut}^+$ such that the orbit of any initial condition in $\mathscr{V}$ converges to $aa$.
 
If $\,\mut \, dS_{Aa,AA}/d\mut(0)<0$, the fixed point
$AA$ is stable and the mutant disappears in the
macroscopic dynamics. 
\end{thm}

\begin{proof}
The result follows immediately 
from Theorem \ref{stabfix}, 
Theorem \ref{wsloc} and Theorem \ref{leve}.
\end{proof}

\bi
The last results of this section concern the proof of Theorem \ref{SKO}. Indeed we want to prove the monotonicity of the function 
\be
\label{effe}
F_{W}(t)=\frac{\langle  M(t)\;,\; W(\mut)\rangle}
{\langle M(t)\;,\; \vec1\rangle}\;.
\ee
Here
 $ M(t)$ denotes a trajectory of the vector field $
X(\mut,\,\cdot\,)$, namely 
$$
\frac{dM}{dt}= X(\mut, M)\;,
$$
in other words $M(t)=\varphi_{t}(M_{0})$, and $ W(\mut)$ is a three
dimensional vector depending continuously on $\mut$.  We denote by $\vec1$
the vector with all components equal to one.  

\begin{prop}\label{monotone}
Assume
$$
\inf_{v\in [-n_{0},n_{0}]}\left|\left\langle
\frac{d\,\Gamma_{0}}{dv}\,,\, W(0)\right\rangle\right|>0\;.
$$
Then for any $|\mut|$ sufficiently small, under the hypothesis of
Theorem \ref{tube}, if $M_{0}$ is close enough to the curve 
$\Gamma_{\mut}$, 
the function $F_{W}(t)$ is strictly monotone. The same
result holds if $ W(0)$ is proportional to $\vec1$ and
$$
\inf_{v\in [-n_{0},n_{0}]}\left|\left\langle
\frac{d\,\Gamma_{0}}{dv}\,,\,
\frac{d W}{d\mut}(0)\right\rangle\right|>0\;.
$$ 
\end{prop}

\begin{proof}
We have 
$$
\frac{dF_{W}}{dt}=\frac{1}{\langle  M(t)\;,\; \vec1\rangle}
\left\langle  X(M)-\frac{\langle  X(M)\,,\,  \vec1\rangle}
{\langle  M(t)\;,\; \vec1\rangle}\;M\,,\,  W(\mut)\right\rangle\;.
$$
Since the invariant curve $\Gamma_{\mut}$ is transversally attracting,
it is enough to consider a point $M\in \Gamma_{\mut}$. If $s$ denotes the
curvilinear abscissa of the curve $\Gamma_{\mut}$, we have  for any $s$
$$
 X\big(\mut,\Gamma_{\mut}(s)\big)=
\big\| X\big(\mut,\Gamma_{\mut}(s)\big)\big\|\;
\frac{d\Gamma_{\mut}}{ds}\;.
$$
Therefore on the invariant curve ($M(t)=\Gamma_{\mut}(s)$ for a certain
$s$ which depends on $t$),
$$
\frac{1}{\langle  M(t)\;,\; \vec1\rangle}
\left\langle  X(M)-\frac{\langle  X(M)\,,\,  \vec1\rangle}
{\langle  M(t)\;,\; \vec1\rangle}\;M\,,\,  W(\mut)\right\rangle
$$
$$
=\frac{\big\| X\big(\mut,\Gamma_{\mut}(s)\big)\big\|}
{\langle \Gamma_{\mut}(s) \;,\; \vec1\rangle}
\left\langle \frac{d\Gamma_{\mut}}{ds}-
\frac{\langle d\Gamma_{\mut}/ds\,,\,  \vec1\rangle}
{\langle \Gamma_{\mut}(s) \;,\; \vec1\rangle}\;\Gamma_{\mut}(s)\,,\, 
W(\mut)\right\rangle\;. 
$$
By Theorem 4.1 in \cite{HPS}  we have
$$
\lim_{\mut\to 0}\frac{d\Gamma_{\mut}}{ds}=\frac{d\Gamma_{0}}{ds}=
\frac{1}{\sqrt{4v^{2}(s)+2n_{0}^{2}}}\;\begin{pmatrix}
v(s)- n_0\cr 
 -2v(s)\cr 
v(s)+n_{0}\cr 
\end{pmatrix}\;,
$$
where 
$$
\frac{dv}{ds}=\frac{1}{\sqrt{4v^{2}(s)+2n_{0}^{2}}}\;.
$$
By a direct computation, one can check that
$$
\lim_{\mut\to0}\left\langle \frac{d\Gamma_{\mut}}{ds}\;,
\; \vec1\right\rangle=0\;,
$$
and the first part of the result follows from Theorem \ref{tube}. 

If $
W(0)=\gamma  \vec1$ for some real number $\gamma$, we have
$$
 W(\mut)=\gamma  \vec1+\mut\;\frac{d
  W}{d\mut}(0)+\mathcal{O}(\mut^{2})\;. 
$$
Therefore
$$
\frac{1}{\langle  M(t)\;,\; \vec1\rangle}
\left\langle  X(M)-\frac{\langle  X(M)\,,\,  \vec1\rangle}
{\langle  M(t)\;,\; \vec1\rangle}\;M\,,\,  W(\mut)\right\rangle
$$
$$
=\frac{\big\| X\big(\mut,\Gamma_{\mut}(s)\big)\big\|}
{\langle \Gamma_{\mut}(s) \;,\; \vec1\rangle}\left(\mut\;
\left\langle \frac{d\Gamma_{\mut}}{ds}-
\frac{\langle d\Gamma_{\mut}/ds\,,\,  \vec1\rangle}
{\langle \Gamma_{\mut}(s) \;,\; \vec1\rangle}\;\Gamma_{\mut}(s)\,,\, 
\frac{d  W}{d\mut}(0)\right\rangle +\mathcal{O}(\mut^{2})\right)\;, 
$$
and the result follows as before.
\end{proof}

\bi
 Consider now  the average 
phenotypic trait $\phi$. This
corresponds to the vector $$
 W_{\phi}(mut)=\left(\begin{array}{c}
\phi\big(u_{A},u_{A}\big)\\
\phi\big(u_{A},u_{a}\big)\\
\phi\big(u_{a},u_{a}\big)\\
\end{array}\right)
=\left(\begin{array}{c}
\phi\big(u_{A},u_{A}\big)\\
\phi\big(u_{A},u_{A}+\mut\big)\\
\phi\big(u_{A}+\mut,u_{A}+\mut\big)\\
\end{array}\right)
$$
$$
=\phi\big(u_{A},u_{A}\big)\left(\begin{array}{c}
1\\
1\\
1\\
\end{array}\right)+
\mut\,\frac{d\phi\big(u_{A},u_{A}\big)}{du_{A}}
\left(\begin{array}{c}
0\\
1/2\\
1\\
\end{array}\right)+\mathcal{O}(\mut^{2})\;.
$$

\begin{cor}
\label{monotonie} 
The function $F_{W_{\phi}}$ is strictly  monotonous for $|\mut|$ small enough.
\end{cor}

\begin{proof}
One gets
$$
\left\langle
\frac{d\,\Gamma_{0}}{dv}\,,\,\frac{d W_{\phi}}{d\mut}(0)\right\rangle=
\left\langle \frac{1}{2n_{0}}\;\begin{pmatrix}
v- n_0\cr 
 -2v\cr 
v+n_{0}\cr 
\end{pmatrix}\;,\;\left(\begin{array}{c}
0\\
1/2\\
1\\
\end{array}\right)\right\rangle=\frac{1}{2}\;,
$$
and by Proposition \ref{monotone} we get the monotonicity in time of the
average phenotypic trait.

\end{proof}

\section{Proof of Theorem \ref{maintheorem}}
\label{prfmt}

\bi
The proof of the theorem will essentially follow the same steps as the ones of  the proof of Theorem 1 in Champagnat \cite{C06} and  of the Appendix A in \cite{CM09}. We will not repeat the details and we will restrict ourselves to the steps that must be modified. The proof is based on intermediary results that we state now.

\begin{prop}
\label{convergence}
Assume that for $K\geq 1$, $Supp(\nu_{0}^K)
 = \{AA, Aa, aa\}$ and $$\lim_{K\to\infty} (\langle \nu_{0}^K, {\mathbbm{1}}_{AA}\rangle,  \langle \nu_{0}^K, {\mathbbm{1}}_{Aa}\rangle, \langle \nu_{0}^K, {\mathbbm{1}}_{aa}\rangle ) =(x_{0},y_{0}, z_{0}) \in V_{\mut}$$ a.s., where $V_{\mut}$ is defined in Theorem \ref{tube}. Then for all $T>0$
 \be
 \label{conv1}
 \lim_{K\to \infty} \sup_{t\in [0,T]}  \Big| \langle \nu_{t}^{\sigma,K}, {\mathbbm{1}}_{AA}\rangle - \varphi_{t}(x_{0},y_{0}, z_{0})_{1}\Big|  = 0 \ a.s,
 \ee
 and similarly for $Aa$ and $aa$, where $\varphi_{t}$ is the flow of the vector field \eqref{leX}.
 \end{prop}
 The proof of this result can be obtained following a standard
 compactness-uniqueness result (see \cite{EK86} or  \cite{FM04}) and
 using Theorem \ref{tube}.

 \begin{prop}
 \label{PGD}  Let $\ Supp(\nu_{0}^K) = \{AA\}$ and let $\tau_{1}$ denote the first mutation time. For any sufficiently small $\varepsilon>0$,
  if  $\langle \nu_{0}^K, {\mathbbm{1}}_{AA}\rangle$ belongs to the ${\varepsilon\over 2}$-neighborhood of $\bar n_{AA}= \frac{f_{AA} - D_{AA}}{C_{AA,AA}}$, the  time of exit of $\langle \nu_{t}^{\sigma,K}, {\mathbbm{1}}_{AA}\rangle$ from the $\varepsilon$-neighborhood of $\bar n_{AA}$ is bigger than $e^{VK} \wedge \tau_{1}$  with probability converging to $1$. 
 
 \noindent Moreover, there exists a constant $c$ such that for any sufficiently small $\varepsilon>0$,the previous result  still holds if the death rate of an individual with genotype $AA$ 
 \be
 \label{de}D_{AA} + C_{AA,AA } \langle \nu_{t}^{\sigma,K}, {\mathbbm{1}}_{AA}\rangle\ee
 is perturbed by an additional random process that is uniformly bounded by $c\, \varepsilon$. 
 \end{prop}
 
 \noindent Such results are standard (cf. \cite{C06}). The first part of this proposition is an exponential deviation estimate on the so-called "exit from an attracting domain" (\cite{FW84}). It is  used to prove  that when the first mutation occurs, the population density has never left  the $\varepsilon$-neighborhood of $\bar n_{AA}$. When a mutation $a$ occurs, the additional term in  \eqref{de} is $C_{AA,Aa } \langle \nu_{t}^{\sigma,K}, {\mathbbm{1}}_{Aa}\rangle + C_{AA,aa } \langle \nu_{t}^{\sigma,K},{\mathbbm{1}}_{aa}\rangle$ which is smaller that $\bar C\,\varepsilon$ if $\langle \nu_{t}^{\sigma,K}, {\mathbbm{1}}_{Aa}\rangle + \langle \nu_{t}^{\sigma,K}, {\mathbbm{1}}_{aa}\rangle\leq \varepsilon$. 
 
 \bi
 From these results, one can  deduce the following proposition, already proved in \cite{C06}.
 
\begin{prop} Let $\ Supp(\nu_{0}^K) = \{AA\}$ and let $\tau_{1}$ denote the first mutation time. There exists $\varepsilon_{0}$ such that if $\langle \nu_{0}^K, {\bf 1}\rangle$ belongs to the $\varepsilon_{0}$-neighborhood of  $\bar n_{AA}$, then for any $\varepsilon
<\varepsilon_{0}$,
\ben
\lim_{{K\to \infty}} \mathbb{P}^K \Big(\tau_{1}> \ln K, \sup_{t\in [\ln K, \tau_{1}]}  | \langle \nu^{\sigma,K}_{t}, 1\rangle - \bar{n}_{AA}|<\varepsilon\Big) = 1,
\een
and 
$K\,\mu_{K }\, \tau_{1}$ converges in law (when $K$ tends to infinity) to a random variable with exponential law with parameter $2 \,f_{AA} \,p_{AA}\,  \bar{n}_{AA}$, that is
for any $t>0$, 
$$\lim_{{K\to \infty}} \mathbb{P}^K \Big(\tau_{1}> {t\over K\mu_{K}}\Big) = \exp(- 2\,p_{AA}\, f_{AA}\, \bar{n}_{AA}\, t).$$ 

\end{prop}

\noindent Then, if 
 $\ \ln K \ll {1\over K\mu_{K}}$,  we deduce that
$\lim_{{K\to \infty}} \mathbb{P}^K\Big(\tau_{1}< \ln K\Big) =0$ and  that for any $\varepsilon>0$
$$\lim_{{K\to \infty}} \mathbb{P}^K \Big(\sup_{t\in [0,\tau_{1}]} | \langle \nu^{\sigma,K}_{t}, 1\rangle - \bar{n}_{AA}|>\varepsilon\Big) = 0.$$

\bi
\noindent Let us define two stopping times which describe the first time where the process arrives in a $\varepsilon$-neighborhood of a stationary state of the dynamical system.

\begin{align}
\label{temps}
\tau_{A}=\tau_{A}(\varepsilon,K)&= \inf\{t\geq 0, \langle
\nu^{\sigma,K}_{t}, \mathbbm{1}_{aa}\rangle = \langle
\nu^{\sigma,K}_{t}, \mathbbm{1}_{Aa}\rangle = 0\, ; \, \langle
\nu^{\sigma,K}_{t}, \mathbbm{1}_{AA}\rangle - \bar{n}_{AA}| <
\epsilon\},\\ \tau_{a}=\tau_{a}(\varepsilon,K)&= \inf\{t\geq 0,
|\langle \nu^{\sigma,K}_{t}, \mathbbm{1}_{aa}\rangle - \bar{n}_{aa}| <
\varepsilon\, ; \, \langle \nu^{\sigma,K}_{t}, \mathbbm{1}_{Aa}\rangle
= \langle \nu^{\sigma,K}_{t}, \mathbbm{1}_{AA}\rangle =0\}.
\end{align}

Note that $\tau_{A}$ is the extinction time of the population with alleles $a$  and fixation of the allele $A$ and that $\tau_{a}$ is the extinction time of the population with allele $A$ and fixation of the  allele $a$. 

\begin{prop}
\label{cv-fitness} Recall that the $S_{Aa,AA}$ has been defined in \eqref{fitness}.
Let $(z_{K})$ be a sequence of integers such that ${z_{K}\over K}$ converges to $\bar{n}_{AA}$. Then
\be 
\label{cv-fitness1}&&\lim_{\varepsilon\to 0}\lim_{{K\to \infty}} \mathbb{P}^K_{{z_{K}\over K} \delta_{AA} +{1\over K}\delta_{Aa}}(\tau_{a}<\tau_{A}) = \, \frac{[S_{Aa,AA}]_+} 
    { f_{Aa}}\\
   && \label{cv-fitness2}\lim_{\varepsilon\to 0}\lim_{{K\to \infty}} \mathbb{P}^K_{{z_{K}\over K} \delta_{AA} +{1\over K}\delta_{Aa}}(\tau_{A}<\tau_{a}) = 1 - \, \frac{[S_{Aa,AA}]_+} 
    { f_{Aa}}\\
  && \label{cv-fitness3} \forall \eta>0,   \lim_{\varepsilon\to 0}\lim_{{K\to \infty}} \mathbb{P}^K_{{z_{K}\over K} \delta_{AA} +{1\over K}\delta_{Aa}}\Big(\tau_{a}\wedge \tau_{A}> {\eta\over K\mu_{K}}\wedge \tau_{1}\Big)=0 
   .\ee
    \end{prop}

\begin{proof} The proof is inspired by the proof of Lemma 3 in \cite{C06}.
We introduce the following stopping times. 
$$R^K_{\varepsilon } = \inf\{t\geq 0\, ; |\langle \nu_{t}^{\sigma,K}, \mathbbm{1}_{AA}\rangle-\bar{n}_{AA}|\geq \varepsilon\},$$
$$S^K_{\varepsilon}=\inf\{ t\geq 0\,; \langle \nu_{t}^{\sigma,K}, \mathbbm{1}_{Aa}\rangle + \langle \nu_{t}^{\sigma,K}, \mathbbm{1}_{aa}\rangle\geq \varepsilon\}.$$
$R^K_{\varepsilon }$ is the time of drift of the resident population $AA$ away from its equilibrium, $S^K_{\varepsilon}$ is the time of invasion of the mutant allele $a$, either if the population with genotype $Aa$ is sufficiently large or the one with genotype $aa$. 

Assume that $\langle \nu_{0}^K, \mathbbm{1}_{Aa}\rangle = {1\over K}$. Using Proposition \ref{PGD}, second part, one can prove as in \cite{C06} that there exist $\rho, V>0$  such that, for $K$ large enough,
\ben\mathbb{P}\left({\rho\over Ku_{K}} <\tau_{1}\right) \geq 1-\varepsilon \, \hbox{ and }\, \mathbb{P}(S^K_{\varepsilon} \wedge \tau_{1}\wedge e^{K V} < R^K_{\varepsilon })\geq 1- \varepsilon.\een

Then, on $[0,\tau_{1}\wedge S_\varepsilon^K\wedge R^K_{\varepsilon}]$, one has $ \bar n_{AA} -  \varepsilon \leq \langle \nu_{t}^{\sigma,K}, \mathbbm{1}_{AA}\rangle\leq  \bar n_{AA} +\varepsilon\ $ and $\ \langle \nu_{t}^{\sigma,K}, \mathbbm{1}_{Aa}\rangle \leq  \varepsilon, \, \langle \nu_{t}^{\sigma,K}, \mathbbm{1}_{aa}\rangle \leq  \varepsilon$. 

\me
Using \eqref{naissance}, \eqref{mort} and  by minorizing or majorizing the birth and death rates,  it can be easily checked that, for $K$ large enough,
almost surely, the process $ (\langle
\nu_t^{\sigma,K},\mathbbm{1}_{\{Aa\}}\rangle, \langle
\nu_t^{\sigma,K},\mathbbm{1}_{\{aa\}}\rangle)$ 
is stochastically lower-bounded and upper-bounded by two normalized
bi-type branching processes ${\Lambda^{1}\over
  K}=(\frac{\Lambda^{11,\varepsilon}_t}{K},
\frac{\Lambda^{12,\varepsilon}_t}{K})_{t\in\mathbb{R}_{+}}$
and ${\Lambda^{2}\over K}=
(\frac{\Lambda^{21,\varepsilon}_t}{K},
\frac{\Lambda^{22,\varepsilon}_t}{K})_{t\in\mathbb{R}_{+}}$.

\me
 The branching processes $\Lambda^{1}$ and $\Lambda^{2}$ have initial condition $(1,0)$  and  birth rates for a state $(y,z)$ of the form (for $i=1,2$),
$$N^i_{{Aa}}(\varepsilon, y,z)=f_{Aa}y + 2 f_{aa} z + o_{1}(\varepsilon)(y+z)\ ;\  N^i_{{aa}}(\varepsilon,y,z)=(f_{Aa}{y\over 2} +  f_{aa} z)\,o_{2}(\varepsilon),$$
and  death rates
$$M^i_{Aa}(\varepsilon,y,z) =  (D_{Aa} + C_{Aa,AA} \bar n_{AA})\, y +
o_{3}(\varepsilon) (y+z)\;,
$$
$$
 M^i_{aa} (\varepsilon,y,z)=  (D_{aa} + C_{aa,AA} \bar n_{AA})\, z  + o_{4}(\varepsilon)(y+z)\;.$$

Moreover we can check that the $o_{i}(\varepsilon)$ don't depend on $K$. 

\me
Let us denote by $q^i_{1}(t)$ and $q^i_{2}(t)$  the probabilities of extinction of the process $\Lambda^i$ before time $t$,  starting respectively from $(1,0)$ or $(0,1)$. These probabilities correspond to the extinction of the allele $a$. Using the generating function, it can be proved (see \cite{AN72}) that the vector $q^i(t)$ is solution of the differential system ${\dot q^i}=  Y^i(\varepsilon, q^i)$ where the vector field $ Y^i$ is of class $C^2$ and 
$$
 Y^i\big(0,(q_{1},q_{2})\big)=\left(\begin{array}{c}
f_{Aa}\, q_{1}^2  + (D_{Aa} + C_{Aa,AA}\, \bar n_{AA}) - (f_{Aa}+
D_{Aa} + C_{Aa,AA}\, \bar n_{AA}) \, q_{1} \nonumber\\
2 f_{aa}\, q_{1}q_{2}  + 
(D_{aa} + C_{aa,AA}\, \bar n_{AA}) - ( 2 f_{aa}+
D_{aa} + C_{aa,AA}\, \bar n_{AA})\, q_{2}
\end{array}\right)\; .
$$
Note that this vector is independent of $i$. 
\end{proof}

\begin{lem} For any $\varepsilon>0$ small enough, we have the following
  properties. 
\begin{enumerate}[i)]
\item The vector field $ Y^i(\varepsilon, \,\cdot\,)$ vanishes at the point
$M_{0}=(1,1)$.
\item If $S_{Aa,AA}<0$, this fixed point is stable, and the
  trajectory emanating from the origin converges to this fixed point.
\item If $S_{Aa,AA}>0$, this fixed point is unstable. There is
  another fixed point
$$
P^i_{\varepsilon}=\left(\begin{array}{c}
\frac{\dst D_{Aa} + C_{Aa,AA}\, \bar n_{AA}}{\dst f_{Aa}}\\
\frac{\dst f_{Aa}\;(D_{aa} + C_{aa,AA}\, \bar n_{AA}) }{\dst
( 2 \;f_{Aa}\;f_{aa}+D_{aa} + C_{aa,AA}\, \bar n_{AA})-2 \;f_{aa}
  \;(D_{Aa} + C_{Aa,AA}\, \bar n_{AA})}
\end{array}\right)+\mathcal{O}^i(\varepsilon)\;,
$$
which is stable and the
  trajectory emanating from the origin converges to this fixed point.
\end{enumerate} 
\end{lem}
\begin{proof}
Assertion \textsl{i)} follows by a direct computation.

The difference between  $ Y^i(\varepsilon, \,\cdot\,)$ and 
$ Y\big(0,\,\cdot\,\big)$ is of order $\varepsilon$
 in $C^{2}$. The first parts of assertions \textsl{ii)} and
 \textsl{iii)} follow 
at once from the similar results for $ Y\big(0,\,\cdot\,\big)$ and
the stability of hyperbolic fixed points (see
for example \cite{gh}).   
Note that in case iii), $$
2 f_{aa}\, q_{1}-( 2 f_{aa}+
D_{aa} + C_{aa,AA}\, \bar n_{AA})<0,
$$
since $q_{1} \in [0,1]$. 

\me
We now prove the second part of case \textsl{ii)}. 
Let $\Phi^{\varepsilon}_{t}$ denote the flow of
the vector field $ Y(\varepsilon, \,\cdot\,)$.
Since the fixed points $M_{0}$ is stable  for $
Y\big(0,\,\cdot\,\big)$, there is a number $r_{0}>0$, such that for any
$\varepsilon>0$ small enough, the ball $B_{r_{0}}(M_{0})$ centered in
$M_{0}$ and of radius $r_{0}$ is attracted to the fixed point $M_{0}$ by
the flow $\Phi^{\varepsilon}_{t}$. Let $T_{0}>0$ denote the smallest time
such that $\Phi^{0}_{t}\big((0,0)\big)\in B_{r_{0}/2}(M_{0})$. This time
is finite since $Y\big(0,(0,0)\big)\neq 0$,
$q_{1}(t)=\Phi^{0}_{t}\big((0,0)\big)_{1}$ converges to $1$ when $t$
tends to infinity, and
$Y_{2}\big(0,(q_{1},q_{2})\big)$ is linear in $q_{2}$. By continuity in
$\varepsilon$ of the map 
$\Phi^{\varepsilon}_{T_{0}}$ (see \cite{gh}), we conclude that for any $\varepsilon>0$
small enough, $\Phi^{\varepsilon}_{T_{0}}\big((0,0)\big)\in
B_{r_{0}}(M_{0})$. The second part of assertion  \textsl{ii)} follows.

\me
The second part of assertion \textsl{iii)} is proved by similar arguments, noting that the fixed point $P^\varepsilon$ depends continuously in $\varepsilon$.
\end{proof}

\me
 We conclude the proof of Proposition \ref{cv-fitness} by similar arguments as in \cite{C06} or in \cite{CM09}, using Theorems \ref{wsloc} and \ref{tube}.

\bi

 \end{document}